\documentclass[11pt,reqno]{amsart}
\usepackage[margin=1in]{geometry}
\usepackage[fleqn,tbtags]{mathtools}
\usepackage[shortlabels]{enumitem}
\usepackage{tikz-cd}
\usepackage{mlmodern}
\usepackage{graphicx}
\usepackage[labelformat=simple]{subcaption}
\usepackage{amssymb}
\usepackage{eucal}
\usepackage{adjustbox}
\usepackage{natbib}

\usepackage{hyperref,xcolor}
\hypersetup{
    colorlinks,
    linkcolor={red!50!black},
    citecolor={blue!50!black},
    urlcolor={blue!80!black}
}

\let\oh=\circ
\newcommand{\ccirc}{\mathbin{\mathchoice
  {\xcirc\scriptstyle}
  {\xcirc\scriptstyle}
  {\xcirc\scriptscriptstyle}
  {\xcirc\scriptscriptstyle}
}}
\newcommand{\xcirc}[1]{\vcenter{\hbox{$#1\oh$}}}
\let\circ\ccirc

\newcommand\restr[2]{{
  \left.\kern-\nulldelimiterspace 
  #1 
  \vphantom{\big|} 
  \right|_{#2} 
  }}

\theoremstyle{definition}
\newtheorem{theorem}{Theorem}[section]
\newtheorem{definition}[theorem]{Definition}
\newtheorem{lemma}[theorem]{Lemma}
\newtheorem{corollary}[theorem]{Corollary}
\newtheorem{proposition}[theorem]{Proposition}

\numberwithin{equation}{section}

\DeclareMathOperator{\diag}{diag}
\DeclareMathOperator{\im}{im}
\DeclareMathOperator{\GL}{\mathsf{GL}}
\DeclareMathOperator{\Sp}{\mathsf{Sp}}
\DeclareMathOperator{\stab}{stab}
\DeclareMathOperator{\Iso}{\mathsf{Iso}}
\let \O \undefined
\DeclareMathOperator{\O}{\mathsf{O}}
\DeclareMathOperator{\U}{\mathsf{U}}

\DeclareMathOperator{\E}{\mathbb{E}}

\let \Re \undefined
\DeclareMathOperator{\Re}{Re}

\DeclareMathOperator{\rank}{rank}
\DeclareMathOperator{\tr}{tr}
\DeclareMathOperator{\Gr}{Gr}
\DeclareMathOperator{\St}{V}
\DeclareMathOperator{\End}{End}
\DeclareMathOperator{\codim}{codim}
\DeclareMathOperator{\Hom}{Hom}

\newcommand{\fb}{{\scriptscriptstyle\mathsf{F}}}

\newcommand{\C}{\mathbb{C}}
\newcommand{\F}{\mathbb{F}}
\newcommand{\B}{\mathcal{B}}
\newcommand{\R}{\mathbb{R}}
\newcommand{\K}{\mathbb{K}}
\newcommand{\V}{\mathbb{V}}
\newcommand{\W}{\mathbb{W}}
\newcommand{\Z}{\mathbb{Z}}
\newcommand{\G}{\mathsf{G}}
\let \H \undefined
\newcommand{\H}{\mathbb{H}}

\title[Indefinite Grassmannian]{The Grassmannian of indefinite subspaces}
\author[L.-H.~Lim]{Lek-Heng~Lim}
\author[R.~Wang]{Rongbiao Thomas Wang}
\author[H.~Yang]{Hongquan Yang}
\address{Computational and Applied Mathematics, University of Chicago, Chicago, IL 60637}
\email{lekheng, rbwang, yanghq@uchicago.edu}

\begin{document}
\begin{abstract}
The homogeneous space $\O_{m,n}(\R)/(\O_{p,q}(\R) \times \O_{m-p,n-q}(\R))$ is an object that has received scant attention, with just two brief mentions in existing literature, and christened the indefinite Grassmannian in one of them. In this article, we develop some of its basic properties, building it from ground up. We will see that, aside from its homogeneous space description, the indefinite Grassmannian may be characterized in several other ways: set-theoretically, it is the manifold of indefinite $(p+q)$-dimensional subspaces in $(m +n)$-dimensional space; it is an adjoint orbit of a Lie group; a semialgebraic smooth manifold of matrices; a base space of a principal bundle whose total space is the indefinite Stiefel manifold, a natural corresponding notion. It may also be naturally equipped with various structures, turning the indefinite Grassmannian into a pseudo-Riemannian manifold; a Einstein manifold; a symplectic manifold; and a pseudo-K\"ahler manifold (last two only over $\C$).  As a centerpiece of this article, we establish two attributes of the indefinite Grassmannian in relation to the standard Grassmannian: (i) any Grassmannian has a Whitney stratification whose highest dimensional strata are indefinite Grassmannians; (ii) the indefinite Grassmannian is a strong deformation retract of a product of two Grassmannians, thereby allowing us to completely ascertain the topology of the former. We will also determine some of the indefinite Grassmannian's features that are within reach --- both geometric (Riemann, Ricci, sectional, and scalar curvatures; second fundamental form) and topological (cohomology ring, homotopy groups, characteristic classes).
\end{abstract}

\maketitle

\section{Introduction}\label{sec:intro}

The Grassmannian is an object with numerous important variations: Lagrangian Grassmannian \cite{Lagrangian}, symplectic Grassmannian \cite{symplectic}, oriented Grassmannian \cite[Chapter~15]{Milnor}, affine Grassmannian \cite{affine}, positive Grassmannian \cite{positive}, and more, each playing important roles in multiple areas of mathematics and physics. In this article we examine yet another variation that, to the best of our knowledge, has been largely overlooked: the indefinite Grassmannian. As an abstract manifold, this is the set of all indefinite $k$-planes in $n$-space --- just as the Lagrangian, symplectic, oriented, affine, and positive Grassmannians comprise respectively Lagrangian, symplectic, oriented, affine, and real $k$-planes with positive Pl\"ucker coordinates in $n$-space.

We encountered this natural object in the course of our recent work \cite[Table~5]{man-rep} where we classified all $\G$-manifolds with a faithful representation; to our surprise, we found only two cursory mentions in existing literature: \cite[Example~4.4.2]{Indef} and \cite[p.~4260]{indef2}, and only in the form of its homogeneous space description. We will see that the indefinite Grassmannian takes many other forms. It is however not a special case of  a generalized Grassmannian in the sense of a quotient of a semisimple Lie group by a parabolic subgroup \cite{flag}, nor can it be realized as an affine Grassmannian in the sense\footnote{Note that this is distinct from the sense used in the first paragraph.} of a flag variety for the loop group \cite{Zhu}; its properties cannot be deduced from these more general objects. 

The main goal of our article is to show that the indefinite Grassmannian possesses rich geometry  and topology similar to its other cousins. 
We will develop the object from ground up, in parallel with standard treatments of the Grassmannian, defining it as an abstract manifold comprising indefinite subspaces of a fixed dimension in a fixed ambient vector space.  We will show that this is indeed equivalent to the homogeneous space characterization in \cite{Indef, indef2}. In addition, we will see that it may also be characterized as an adjoint orbit of a Lie group, and as a smooth semialgebraic manifold of matrices, and can be defined over $\R$, $\C$, $\H$ alike. The definition of an indefinite Grassmannian immediately leads us to an indefinite Stiefel manifold --- the latter is a principal bundle over the former. Over $\R$, $\C$, $\H$, the indefinite Grassmannian has a natural pseudo-Riemannian metric; and over $\C$, it also has a symplectic form and a complex structure that are compatible with the pseudo-Riemannian metric in a way that makes the complex indefinite Grassmannian a pseudo-K\"ahler manifold.  With respect to this metric, we compute its intrinsic curvatures --- Riemann, Ricci, sectional, scalar, and traceless Ricci; the last of which is identically zero, showing that the indefinite Grassmannian is an Einstein manifold. We will also provide the second fundamental form of its natural embedding as a manifold of matrices.

In our opinion, the most striking property of an indefinite Grassmannian is its relation with respect to a standard Grassmannian: Just as a (standard) Grassmannian  $\Gr_k(\mathbb{V})$ has a well-known CW decomposition into Schubert cells, we will show that it also has a Whitney stratification whose highest dimensional strata are the indefinite Grassmannians. The former is algebraic geometric in nature, Schubert cells being locally closed subvarieties of $\Gr_k(\mathbb{V})$; the latter is differential geometric in nature, indefinite Grassmannians being open submanifolds of $\Gr_k(\mathbb{V})$. They are also entirely distinct; neither can be obtained from the other. The lower-dimensional strata in this Whitney stratification may be viewed as a type of ``degenerate Grassmannians'' and appear interesting in their own right; we will study their normal bundles as embedded submanifolds of  $\Gr_k(\mathbb{V})$.

Topologically, we will see that the indefinite Grassmannian is a strong deformation retract of a product of two standard Grassmannians, whose dimensions are determined by the indefinite form that defines the indefinite Grassmannian. This allows us to completely determine the topology of the indefinite Grassmannian from the signature of its indefinite form, allowing us to deduce the cohomology, homotopy, and characteristic classes of the indefinite Grassmannian from those of the Grassmannian.

\section{Notations}\label{sec:note}

In this article, we consider $d$-dimensional vector spaces $\V$ over $\F \coloneqq \R$, $\C$, or $\H$. When $\F = \H$, we assume that $\V$ is a right $\H$-vector space. If $\V$ is equipped with the Euclidean inner product, then $\V \cong \F^d$ and the group of isometries of $\V$ is isomorphic to 
\[
 \Iso(\V) \cong \Iso_d(\F) \coloneqq \{Q \in \F^{d \times d}: Q^* Q =I\}  = \begin{cases}
    \O_d &\text{if } \F = \R,\\
    \U_d &\text{if } \F = \C,\\
    \Sp_d &\text{if } \F = \H.
\end{cases}
\]
We write $e_i \in \F^d$ for its $i$th elementary basis vector, $i=1,\dots,d$. We write $P_\W \in \F^{d \times d}$ for the orthogonal projector onto $\W \subseteq \F^d$ represented in the standard basis. We use the term `projector' to mean `projection matrices.'

\subsection{Grassmannian} Let $k<d$. We regard the Grassmannian of $k$-dimensional subspaces ($k$-planes for short) in $\V$ as an abstract manifold defined set theoretically by
\[
\Gr_k(\V) = \{\W \subseteq \V: \dim_\F \W = k\}.
\]
It has a well-known characterization as a homogeneous space
\[
\Gr_k(\V) \cong \Iso_d(\F)/\bigl( \Iso_k(\F) \times \Iso_{d-k}(\F)\bigr).
\]
Less well-known is its representation as a submanifold of $\F^{d \times d}$ \cite{LLY2020,LY2024,man-rep}: For any $\lambda  \ne \mu  \in \R$,
\begin{equation}\label{eq:Griv}
\Gr_k(\V) \cong \{Q\diag(\lambda I_k,\mu I_{d-k})Q^{-1} \in \F^{d \times d} :Q \in \Iso_d(\F)\} \eqqcolon \Gr_k^{\lambda ,\mu }(\F^d).
\end{equation}
This is an adjoint orbit of $\Iso_d(\F)$ but for specific values of $\lambda $ and $\mu $, it often has a simple algebraic description. For example, with $(\lambda ,\mu ) = (1,0)$, we get its other well-known representation \cite[Example~1.2.22]{Nic} as a set of orthogonal projectors,
\begin{equation}\label{eq:Grpr}
\Gr_k(\V) \cong \Gr_k^{1,0}(\F^d)  = \{P \in \F^{d \times d}: P = P^2 = P^*,\, \tr(P) = k\} \eqqcolon \Gr_k^\pi (\F^d) .
\end{equation}

\subsection{Indefinite linear algebra}\label{sec:ila}

To define an indefinite Grassmannian, we need to equip $\V$ with an indefinite Hermitian form $\omega : \V \times \V \to \R$. Recall that the signature of $\omega$ is the triple $\mathsf{s}(\omega) = (m,n,r)$ where
\begin{align*}
m &\coloneqq \max \{ \dim_\F \W : \W \subseteq \V, \, \omega(v,v) > 0 \text{ for all nonzero } v \in \W  \},\\
n &\coloneqq \max \{ \dim_\F \W : \W \subseteq \V, \, \omega(v,v) < 0 \text{ for all nonzero } v \in \W \},\\
r &\coloneqq \dim \{v \in \V: \omega(v,w) = 0 \text{ for all } w \in \V\};
\end{align*}
and that an isometry between $(\V_1,\omega_1)$ and $(\V_2, \omega_2)$ is a linear isomorphism $Q: \V_1 \to \V_2$ such that $\omega_1(x,y) = \omega_2(Qx, Qy)$ for all $x$, $y \in \V_1$. 

Let $\mathsf{H}^2(\V)$ denote the space of Hermitian operators on an $\F$-vector space $\V$ with respect to the Euclidean inner product. We also denote the inertia of $H \in \mathsf{H}^2(\V)$ by $\mathsf{s}(H) = (m,n,r)$, recalling that here $m$, $n$, $r$ are respectively the number of positive, negative, and zero eigenvalues counted with multiplicities. This slight abuse notation is inconsequential because of the following. We write $\mathsf{H}^2(\F^d)$ for the set of $d \times d$ Hermitian matrices over $\F$ and
\[
 \mathsf{H}^2(\F^d; B) \coloneqq \{ X \in \F^{d \times d} : BX = X^* B \}
\]
for the set of $B$-Hermitian matrices for a $B \in \mathsf{H}^2(\F^d)$.

For a basis $\mathcal{B} = \{v_1, \dots, v_d\}$ of $\V$, the Gram matrix $B_\omega \in \F^{d \times d}$ of $\omega$ is the matrix with $(i,j)$th entry $\omega(v_i,v_j)$, $i,j = 1,\dots,d$. Clearly,
\[
\omega(x,y) = [\overline{x}_1,\dots,\overline{x}_d] B_\omega \begin{bmatrix}
    y_1 \\ \vdots \\y_d
\end{bmatrix} \quad \text{if } x = \sum_{i=1}^d v_i x_i, \quad y = \sum_{i=1}^d v_i y_i.
\]
The map $\omega \mapsto B_\omega$ gives a bijection from the set of indefinite forms over $\V$ to  $\mathsf{H}^2(\F^d)$. The Sylvester's law of inertia  guarantees that the signature of a Hermitian form is the inertia of its Gram matrix, irrespective of the choice of basis. For easy reference, and given that the quaternionic case \cite[Theorem~4.1.6]{quaternion} is relatively obscure, we state it \cite[Theorem~4.5.8]{HornJohnson2012} formally below.
\begin{theorem}[Sylvester's law of inertia]
If $H \in \mathsf{H}^2(\F^d)$, then $\mathsf{s}(H) = \mathsf{s}(SHS^*)$ for all $S \in \GL_d(\F)$ and $\mathsf{s}(H) = (m,n,r)$ if and only if $H = S \diag(I_m,-I_n,0)S^*$ for some $S \in \GL_d(\F)$. Thus  $\mathsf{s}(\omega) = \mathsf{s}(B_\omega) $ for an Hermitian form $\omega$,  irrespective of the choice of basis.
\end{theorem}

In the rest of this article, we will assume that $\omega$ is nondegenerate, i.e., $r=0$ in its signature, and set
\[
I_{m,n} \coloneqq \diag(I_m,-I_n), \qquad I_{m,n}^{p,q} = I_{p,q} \oplus I_{m-p,n-q}
\]
for any $p=0,\dots, m$ and $q = 0,\dots, n$. Note also that we use both $A \oplus B$ and $\diag(A, B)$ for block diagonal matrices.
By Sylvester's law of inertia, without loss of generality, we may assume  (i) the standard basis on $\F^d$; (ii) $\omega$ takes the form
\[
\omega_{m,n}(x,y) \coloneqq \overline{x}_1y_1 + \dots + \overline{x}_my_m - \overline{x}_{m+1}y_{m+1}-\dots-\overline{x}_dy_d;
\]
and (iii) $B_{\omega} = I_{m,n}$. With this, $\mathsf{H}^2(\F^d; \omega_{m,n}) = \mathsf{H}^2(\F^d)$.

For a nondegenerate $\omega$, the orthogonal complement of a subspace $\W$ with respect to $\omega$ is 
\[
\W^{\perp_\omega} \coloneqq \{v \in \V: \omega(v,w) = 0 \text{ for all } w \in \W\};
\]
and the group of isometries of $(\V,\omega)$ is
\[
\Iso_{\omega}(\V) \cong \Iso_{m,n}(\F) \coloneqq \{Q \in \F^{d \times d}: Q^* I_{m,n} Q = I_{m,n}\}=\begin{cases}
    \O_{m,n} &\text{if } \F = \R,\\
    \U_{m,n}  &\text{if } \F = \C,\\
    \Sp_{m,n}  &\text{if } \F = \H.
\end{cases} 
\]
In any basis, up to an inner automorphism,
\[
\Iso(\F; B_\omega) \coloneqq \{Q \in \F^{d \times d}: Q^* B_\omega Q = B_{\omega}\} \cong \Iso_{m,n}(\F).
\]

For $X\in\Hom_{\F}(\W,\W^{\perp_\omega})$, its $\omega$-adjoint is the unique
$X^\dagger \in\Hom_{\F}(\W^{\perp_\omega},\W)$ satisfying
\[
    \omega(Xu,v)=\omega(u,X^\dagger v)
    \qquad \text{for all } u\in \W, \; v\in \W^{\perp_\omega}.
\]
In a fixed $\omega$-orthonormal basis, $X^\dagger =I_{p,q}X^*I_{m-p,n-q}$. 

\section{Indefinite Grassmannian}\label{sec:basic}

In the rest of the article, we assume that
\[
p \leq m, \quad q \leq n, \quad d \coloneqq m+n, \quad k \coloneqq p+q.
\]
We begin by defining the indefinite Grassmannian as an abstract manifold.

\begin{definition}\label{def:indef}
Let $\V$ be a $d$-dimensional $\F$-vector space with a nondegenerate indefinite form $\omega$ where $\mathsf{s}(\omega) = (m,n,0)$. The \emph{indefinite Grassmannian} of signature $(p,q)$ is the set of all $k$-dimensional subspaces $\W$ with $\mathsf{s}(\omega \vert_\W) = (p,q,0)$, i.e.,
\[
\Gr_{p,q}(\V) \coloneqq \{\W \subseteq \V: \dim_\F(\W) = k,\, \mathsf{s}(\omega \vert_\W) = (p,q,0)\}.
\]
\end{definition}

While Definition~\ref{def:indef} appears to depend on a choice of $\omega$, it is only dependent on its signature, justifying the notation:
\begin{proposition}\label{prop:indp}
If $\omega_1$ and $\omega_2$ are two indefinite forms over $\V$ with $\mathsf{s}(\omega_1) = \mathsf{s}(\omega_2) = (m,n,0)$, then there exists an isometry $Q:(\V,\omega_1) \to (\V,\omega_2)$ bijectively mapping $\Gr_{p,q}(\V, \omega_1)$ to $\Gr_{p,q}(\V,\omega_2)$.
\end{proposition}
\begin{proof}
By Sylvester's law of inertia, there exist bases $\mathcal{B}_1$ and $\mathcal{B}_2$ such that $B_{\omega_1} = B_{\omega_2} = I_{m,n}$ in the respective basis. Let $Q$ be the change-of-basis matrix taking $\mathcal{B}_1$ to $\mathcal{B}_2$.
\end{proof}

Another consequence of Proposition~\ref{prop:indp} is that we may restrict our attention to
\[
\Gr_{p,q}(\F^{m+n}) \coloneqq \Gr_{p,q}(\F^{m+n}, \omega_{m,n})
\]
without any loss of generality; and we shall do so in the rest of this article. We next show that the indefinite Grassmannian is a submanifold of the standard Grassmannian, and is a complex submanifold when both are complex manifolds.

\begin{proposition}[Indefinite Grassmannian as open submanifold]\label{prop:open}
Over $\F$, $\Gr_{p,q}(\F^{m+n})$ is an open smooth real submanifold of $\Gr_k(\F^d)$. In addition, $\Gr_{p,q}(\C^{m+n})$ is a complex open submanifold of $\Gr_k(\C^d)$. In particular,
\begin{equation}\label{eq:dim}
    \dim_\R \Gr_{p,q}(\F^{m+n}) = \dim_\R \Gr_k (\F^d)  = k(d-k) \cdot \dim_\R \F.
\end{equation}
\end{proposition}

\begin{proof}
Let $\W \in \Gr_k(\F^d)$ and $P_\W$ be the Euclidean orthogonal projector onto $\W$. The map $\W \mapsto P_\W$ is a real analytic diffeomorphism $\Gr_k(\F^d) \cong \Gr_k^\pi (\F^d)$ \cite{LLY2020}. Consider the continuous map
\[
\Psi \colon \Gr_k(\F^d) \to \mathsf{H}^2(\F^d), \quad \W \mapsto P_\W I_{m,n} P_\W + (I - P_\W).
\]
Let $Q \in \Iso_d(\F)$ be a matrix whose first $k$ columns form an orthonormal basis of $\W$ with respect to the standard inner product. Then
\[
Q^* P_\W Q = I_k \oplus 0_{d-k} \qquad \text{and} \qquad
Q^* I_{m,n} Q = \begin{bmatrix} B_{11} & B_{12} \\ B_{12}^* & B_{22} \end{bmatrix},
\]
where $B_{11} \in \mathsf{H}^2(\F^k)$ is the Gram matrix of $\omega_{m,n}\vert_\W$. Thus
\[
Q^* \Psi(\W) Q = (I_k \oplus 0_{d-k})(Q^* I_{m,n} Q)(I_k \oplus 0_{d-k}) + (0_k \oplus I_{d-k}) = B_{11} \oplus I_{d-k}.
\]
By Sylvester's law of inertia,
\[
\mathsf{s}(\Psi(\W)) = \mathsf{s}(B_{11} \oplus I_{d-k}) = \mathsf{s}(\omega_{m,n}\vert_\W) + (d-k,0,0).
\]
Hence $\W \in \Gr_{p,q}(\F^{m+n})$ if and only if $\mathsf{s}(\Psi(\W)) = (p+d-k,\,q,\,0)$, i.e.,
\[
\Gr_{p,q}(\F^{m+n}) = \Psi^{-1}(\Sigma), \quad \Sigma \coloneqq \{H \in \mathsf{H}^2(\F^d) : \mathsf{s}(H) = (p+d-k,\,q,\,0)\}.
\]
Clearly, $\Sigma \subseteq \GL_d(\F)$ and it is open as eigenvalues of $H$ depend continuously on $H$. Therefore, $\Gr_{p,q}(\F^{m+n}) = \Psi^{-1}(\Sigma)$ is open in $\Gr_k(\F^d)$ and thus an open smooth submanifold. If $\F = \C$, it is an open complex submanifold. 
\end{proof}

\section{Three models of the indefinite Grassmannian}\label{sec:many}

In addition to the set theoretic Definition~\ref{def:indef}, we provide three other models for the indefinite Grassmannian: as a homogeneous space, as a semialgebraic manifold of orthogonal projectors, and as an $\Iso_{m,n}(\F)$-orbit in the space of $I_{m,n}$-Hermitian matrices. In the last case, one may also choose appropriate parameters to view the indefinite Grassmannian as a set of $I_{m,n}$-orthogonal projectors. As in the case of the standard Grassmannian, the different ways to regard an indefinite Grassmannian afford a versatility when working with it --- all three models would be useful in due course.

We begin by showing that the homogeneous space $\Iso_{m,n}(\F)/(\Iso_{p,q}(\F) \times \Iso_{m-p,n-q}(\F))$ in \cite{Indef, indef2} is indeed the indefinite Grassmannian as defined in Definition~\ref{def:indef}.
\begin{proposition}[Indefinite Grassmannian as homogeneous space]\label{prop:hom}
The group $\Iso_{m,n}(\F)$ acts transitively on $\Gr_{p,q}(\F^{m+n})$, giving an $\Iso_{m,n}(\F)$-equivariant diffeomorphism
\[
\Gr_{p,q}(\F^{m+n}) 
\; \cong \; 
\Iso_{m,n}(\F)/\bigl(\Iso_{p,q}(\F) \times \Iso_{m-p,n-q}(\F) \bigr).
\]
\end{proposition}

\begin{proof}
Let $\W$, $\W' \in \Gr_{p,q}(\F^{m+n})$. By Sylvester's law of inertia, there are bases $u_1,\dots,u_k \in \W$ and $u_1',\dots,u_k' \in \W'$ such that $\omega \vert_{\W} =\omega \vert_{\W'} =I_{p,q}$ in the respective bases. The map $\varphi \colon \W \to \W'$ given by $\varphi(u_i) = u_i'$, $i = 1,\dots,k$ is an $\omega$-isometry. By Witt's extension theorem (see \cite[Theorem~3.9]{Artin1957} for $\R$ and $\C$, and \cite[Theorem~5.1]{H-Witt} for $\H$), $\varphi$ extends to an $\omega$-isometry $\overline{\varphi} \colon \F^{m+n} \to \F^{m+n}$ with $\overline{\varphi}(\W) = \W'$. Therefore, the action $\Iso_{m,n}(\F) \times  \Gr_{p,q}(\F^{m+n})$, $(Q , \W ) \mapsto Q \W$ is transitive. Let $\mathsf{H}$ be the stabilizer of $\W$ in $\Iso_{m,n}(\F)$. The map
\[
\rho: \mathsf{H} \to \GL(\W) \times \GL(\W^{\perp_\omega}), \quad g \mapsto (g\vert_\W, g\vert_{\W^{\perp_\omega}})
\]
is well-defined as $Q(\W^{\perp_\omega}) = Q(\W)^{\perp_\omega} = \W^{\perp_\omega}$ for all $Q \in \mathsf{H}$; it is also a Lie group homomorphism. Since $Q$ is an $\omega$-isometry, $Q\vert_{\W}$ and $Q\vert_{\W^{\perp_\omega}}$ are isometries with respect to restrictions of $\omega$. As $\mathsf{s}(\omega\vert_{\W^{\perp_\omega}})= (m,n,0) - \mathsf{s} (\omega\vert_{\W})  = (m-p,n-q,0)$, we have that $Q\vert_{\W} \in \Iso_{p,q}(\F)$ and $Q\vert_{\W^{\perp_\omega}} \in \Iso_{m-p,n-q}(\F)$, and thus $\im(\rho) \subseteq  \Iso_{p,q}(\F) \times \Iso_{m-p,n-q}(\F)$. On the other hand, for $(Q_1,Q_2) \in \Iso_{p,q}(\F) \times \Iso_{m-p,n-q}(\F)$, the isometry $Q \coloneqq Q_1 \oplus Q_2$ on $\F^d = \W \oplus \W^{\perp_\omega}$ satisfies $\rho(Q) = (Q_1,Q_2)$. Finally, $\rho$ is injective because $g \in \mathsf{H}$ is completely determined by $g\vert_\W$ and $g\vert_{\W^{\perp_\omega}}$. Hence $\rho$ defines an isomorphism of Lie groups $\mathsf{H} \cong \Iso_{p,q}(\F) \times \Iso_{m-p,n-q}(\F)$. The orbit--stabilizer theorem yields the required result.
\end{proof}

In case it is not clear, $\Iso_{p,q}(\F) \times \Iso_{m-p,n-q}(\F)$ is not a parabolic subgroup of $\Iso_{m,n}(\F)$ and so $\Gr_{p,q}(\F^{m+n})$ is not a generalized Grassmannian, i.e., $\G/\mathsf{P}$ where $\mathsf{P}$ is a maximal parabolic subgroup. 

We next characterize the indefinite Grassmannian as a semialgebraic manifold of orthogonal projectors a la \eqref{eq:Grpr}.
\begin{proposition}[Indefinite Grassmannian as orthogonal projectors]\label{prop:proj}
Let
\[
    \Gr_{p,q}^\pi (\F^{m+n}) \coloneqq \{ P \in \F^{d \times d} : P = P^2 = P^*, \,\tr(P) = k, \, \mathsf{s}(PI_{m,n}P) = (p,q, d-k) \} \subseteq \Gr_k^\pi (\F^d).
\]
Let $\psi \colon \Gr_k(\F^d) \to \Gr_k^\pi (\F^d)$, $\W \mapsto P_\W$, and $\varphi = \psi \vert_{\Gr_{p,q}(\mathbb{F}^{m+n})}$.  Then, in the category of smooth manifolds, the following diagram commutes
\begin{equation}\label{eq:proj-model-commute}
\begin{tikzcd}[column sep=large, row sep=large]
\Gr_{p,q}(\mathbb{F}^{m+n}) \arrow[r, hook] \arrow[d, "\varphi"', "\cong"]
& \Gr_k(\mathbb{F}^{d}) \arrow[d, "\psi"', "\cong"] \\
\Gr_{p,q}^\pi (\mathbb{F}^{m+n}) \arrow[r, hook]
& \Gr_k^\pi (\mathbb{F}^{d}).
\end{tikzcd}
\end{equation}
\end{proposition}
\begin{proof}
We show that $\varphi$ is a bijection by showing that for a $k$-dimensional subspace $\W$ of $\F^d$, $\W \in \Gr_{p,q}(\F^{m+n})$ if and only if $\mathsf{s}(P_\W I_{m,n}P_\W)=(p,q, d-k)$. If so, $\psi^{-1}\bigl(\Gr_{p,q}^\pi (\F^{m+n})\bigr) = \Gr_{p,q}(\F^{m+n})$ and the bijection $\psi$ restricts to a bijection $\varphi$ onto $\Gr_{p,q}^\pi (\F^{m+n})$; furthermore, $\varphi$ is a diffeomorphism onto its image as it is a restriction of a diffeomorphism \cite{LLY2020} $\psi$   to an open submanifold.

Let $\W \in \Gr_k(\F^d)$ and $Q \in \Iso_d(\F)$ be such that its first $k$ columns $q_1,\dots,q_k$ form an orthonormal basis of $\W$ in the Euclidean inner product. Then $Q^* P_\W Q = I_k \oplus 0_{d-k}$. Partition
\[
Q^* I_{m,n} Q = \begin{bmatrix}
    B_{11} & B_{12} \\
    B_{12}^* & B_{22}
\end{bmatrix}, \quad  B_{11} \in \mathsf{H}^2(\F^{k}),\; B_{12} \in \F^{k \times (d-k)},\; B_{22} \in \mathsf{H}^2(\F^{d-k}).
\]
So $B_{11}$ is the matrix of $\omega_{m,n}\vert_\W$ in the basis $q_1,\dots, q_k$ and so $\mathsf{s}(\omega_{m,n}\vert_\W) = \mathsf{s}(B_{11})$.
Since
\[
Q^* P_\W I_{m,n} P_\W Q = (Q^* P_\W Q)(Q^* I_{m,n} Q)(Q^* P_\W Q) = B_{11} \oplus 0_{d-k},
\]
Sylvester's law of inertia gives
\[
\mathsf{s}(P_\W I_{m,n} P_\W) = \mathsf{s}(B_{11} \oplus 0_{d-k}) = \mathsf{s}(B_{11}) + (0,0,d-k) = \mathsf{s}(\omega_{m,n}\vert_\W) + (0,0,d-k).
\]
Therefore, $\mathsf{s}(P_\W I_{m,n} P_\W) = (p,q,d-k)$ if and only if $\mathsf{s}(\omega_{m,n}\vert_\W) = (p,q,0)$, i.e., if and only if $\W \in \Gr_{p,q}(\F^{m+n})$.
\end{proof}

\begin{corollary}[Indefinite Grassmannian is semialgebraic]\label{cor:semialgebraic}
$\Gr_{p,q}^\pi (\F^{m+n})$ is a semialgebraic smooth submanifold of $\F^{d\times d}$. 
\end{corollary}

\begin{proof}
Since $\{(H,S): H = SI_{m,n}S^*,\, S \in \GL_d(\F)\}$ is a semialgebraic set, the result follows from applying  Sylvester's law of inertia and Tarski--Seidenberg  theorem to Proposition~\ref{prop:proj}.
\end{proof}
    
The embedding in \eqref{eq:proj-model-commute} has one shortcoming, namely, it does not preserve $\Iso_{m,n}(\F)$ symmetry as in Proposition~\ref{prop:hom}. We now seek an $\Iso_{m,n}(\F)$-equivariant embedding of $\Gr_{p,q}(\F^{m+n})$. 
\begin{proposition}\label{prop:equi-orb}
Let $\lambda  \ne \mu  \in \R$. Then
\begin{equation}\label{eq:equiv-model}
\begin{aligned}
    \Gr_{p,q}(\F^{m+n}) &\cong \{ Q \diag(\lambda  I_p, \mu  I_{m-p}, \lambda  I_q, \mu  I_{n-q}) Q^{-1} \in \mathsf{H}^2(\F^d; I_{m,n}) : Q \in \Iso_{m,n}(\F) \}\\
    &\cong \{ Q(\lambda  I_k \oplus \mu  I_{d-k})Q^{-1} \in \mathsf{H}^2(\F^d; I_{m,n}^{p,q}) : Q \in \Iso(\F; I_{m,n}^{p,q}) \} \eqqcolon \Gr^{\lambda , \mu }_{p,q}(\F^{m+n}).
\end{aligned}
\end{equation}
\end{proposition}
\begin{proof}
Consider the permutation matrix 
\begin{equation}\label{eq:perm}
    \Pi \coloneqq \bigl[ e_1, \dots , e_p , e_{m+1}, \dots , e_{m+q} , e_{p+1}, \dots, e_m, e_{m+q+1}, \dots , e_d \bigr].
\end{equation}
Then
$\Pi^* \diag(\lambda  I_p, \mu  I_{m-p}, \lambda  I_q, \mu  I_{n-q}) \Pi = \lambda  I_k \oplus \mu  I_{d-k}$, $\Pi^* I_{m,n} \Pi = I_{m,n}^{p,q}$, and $\Pi^* \Iso_{m,n}(\F)\Pi = \Iso(\F; I_{m,n}^{p,q})$. Let $\W$ be the span of the first $k$ columns of $\Pi$ and $\W'$ the span of the last $d-k$. Then $\mathsf{s}(\omega_{m,n}\vert_{\W}) = \mathsf{s}(I_{p,q}) = (p,q,0)$, so $\W \in \Gr_{p,q}(\F^{m+n})$, $\mathsf{s}(\omega_{m,n}\vert_{\W'}) = (m-p,n-q,0)$, and $\W' = \W^{\perp_{\omega_{m,n}}}$. Since $\Pi$ is unitary,
$Q \diag(\lambda  I_p, \mu  I_{m-p}, \lambda  I_q, \mu  I_{n-q})Q^{-1} = \diag(\lambda  I_p, \mu  I_{m-p}, \lambda  I_q, \mu  I_{n-q})$ if and only if  $(\Pi^* Q \Pi) (\lambda  I_k \oplus \mu  I_{d-k}) = (\lambda  I_k \oplus \mu  I_{d-k}) (\Pi^* Q \Pi)$. Thus, it suffices to prove the second diffeomorphism.
We partition $Q \in \Iso(\F;I_{m,n}^{p,q})$ into $Q = \begin{bsmallmatrix} Q_{11} & Q_{12} \\ Q_{21} & Q_{22} \end{bsmallmatrix}$ with $Q_{11} \in \F^{k \times k}$ and $Q_{22} \in \F^{(d-k) \times (d-k)}$.
Then
\[
Q \begin{bmatrix} \lambda  I_k & 0 \\ 0 & \mu  I_{d-k} \end{bmatrix} = \begin{bmatrix} \lambda  Q_{11} & \mu  Q_{12} \\ \lambda  Q_{21} & \mu  Q_{22} \end{bmatrix},
\qquad
\begin{bmatrix} \lambda  I_k & 0 \\ 0 & \mu  I_{d-k} \end{bmatrix} Q = \begin{bmatrix} \lambda  Q_{11} & \lambda  Q_{12} \\ \mu  Q_{21} & \mu  Q_{22} \end{bmatrix},
\]
which are equal if and only if $(\lambda  - \mu )Q_{12} = 0$ and $(\lambda  - \mu )Q_{21} = 0$. Since $\lambda  \ne \mu $, we have $Q_{12}=0$ and $Q_{21}=0$. Conversely, any block-diagonal $Q$ commutes with $\lambda  I_k \oplus \mu  I_{d-k}$. Therefore,
\begin{align*}
\mathsf{H} &\coloneqq \{Q \in \Iso(\F;I_{m,n}^{p,q}) : Q(\lambda  I_k \oplus \mu  I_{d-k})Q^{-1} = \lambda  I_k \oplus \mu  I_{d-k}\} \\
&= \biggl\{Q \in \Iso(\F;I_{m,n}^{p,q}) : Q = \begin{bmatrix} Q_{11} & 0 \\ 0 & Q_{22} \end{bmatrix} \biggr\}.
\end{align*}
Now $Q_{21} = 0$ if and only if $Q\W \subseteq \W$; since $Q$ is invertible, $Q\W = \W$. If $Q\W = \W$, then $Q\W'  = (Q\W)^{\perp_{\omega_{m,n}}} = \W'$, i.e., $Q_{12} = 0$. Therefore, $\mathsf{H}$ is the stabilizer of $\W$ in $\Iso_{m,n}(\F)$. For $Q \in \mathsf{H} \subseteq \Iso_{m,n}(\F)$,
\[
Q^* I_{m,n}^{p,q}Q =\begin{bmatrix}
    Q_{11}^* &0\\ 0 &Q_{22}^*
\end{bmatrix} \begin{bmatrix}
    I_{p,q} &0 \\0 &I_{m-p,n-q}
\end{bmatrix}\begin{bmatrix}
    Q_{11} &0\\ 0 &Q_{22}
\end{bmatrix} = \begin{bmatrix}
    I_{p,q} &0 \\0 &I_{m-p,n-q}
\end{bmatrix},
\]
i.e., $Q_{11} \in \Iso_{p,q}(\F)$ and $Q_{22} \in \Iso_{m-p,n-q}(\F)$. By the orbit--stabilizer theorem and Proposition~\ref{prop:hom},
\[
\Gr^{\lambda , \mu }_{p,q}(\F^{m+n}) \cong \Iso(\F;I_{m,n}^{p,q})/\mathsf{H} = \Iso_{m,n}(\F)/\bigl(\Iso_{p,q}(\F) \times \Iso_{m-p,n-q}(\F)\bigr) \cong \Gr_{p,q}(\F^{m+n}). 
\]
Lastly, for all $Q \in \Iso(\F;I_{m,n}^{p,q})$, a direct computation shows
\begin{equation}\label{eq:i-herm}
(Q (\lambda  I_k \oplus \mu  I_{d-k}) Q^{-1})^* I_{m,n}^{p,q}  = I_{m,n}^{p,q} (Q (\lambda  I_k \oplus \mu  I_{d-k}) Q^{-1}),
\end{equation}
as required.
\end{proof}
A projection $P \in \F^{d \times d}$ is orthogonal with respect to the indefinite form defined by $I_{m,n}$ if it is $I_{m,n}$-Hermitian, i.e.,
\[
P^2 = P \quad \text{and} \quad P^* I_{m,n} = I_{m,n}P.
\]
We call these $I_{m,n}$-orthogonal projector, a consequence of
\begin{equation}\label{eq:i-ortho}
    \ker(P) = \im(P)^{\perp_{I_{m,n}}}.
\end{equation}
Choosing $\lambda  = 1$ and $\mu  = 0$ in \eqref{eq:equiv-model} yields an alternative to Proposition~\ref{prop:proj}.

\begin{proposition}[Indefinite Grassmannian as $I_{m,n}$-orthogonal projectors]\label{prop:i-proj}
\[
\Gr_{p,q}(\F^{m+n}) \cong \{P \in \F^{d \times d}: P = P^2,\,P^* I_{m,n} = I_{m,n}P, \,\mathsf{s}(I_{m,n}P) = (p,q,d-k)\}.
\]
\end{proposition}
\begin{proof}
Let $D = \diag(I_p, 0_{m-p}, I_q, 0_{n-q})$. For $Q \in \Iso_{m,n}(\F)$,  $(QDQ^{-1})(QDQ^{-1}) = QDQ^{-1}$. By \eqref{eq:i-herm}, $QDQ^{-1}$ is an $I_{m,n}$-orthogonal projector. As $I_{m,n} QDQ^{-1} = \bigl( (Q^{-1})^* I_{m,n} Q^{-1} \bigr) Q D Q^{-1} = (Q^{-1})^* (I_{m,n} D) Q^{-1}$,
\[
\mathsf{s}(I_{m,n} QDQ^{-1}) = \mathsf{s}(I_{m,n}D) = \mathsf{s}\bigl(\diag(I_p, 0_{m-p}, -I_q, 0_{n-q})\bigr) = (p,q,d-k)
\]
by Sylvester's law of inertia.
On the other hand, let $P \in \F^{d \times d}$ be an $I_{m,n}$-orthogonal projector. Then $I_{m,n}P \in \mathsf{H}^2(\F^d)$. Suppose $\mathsf{s}(I_{m,n}P)=(p,q,d-k)$. Let $S = [S_1, S_2] \in \F^{d \times d}$ be such that the columns of $S_1 \in \F^{d \times k}$ form a basis of $\im(P)$ and those of $S_2 \in \F^{d \times (d-k)}$ form a basis of $\ker(P)$. Since $P=P^2$, $\F^d = \im(P) \oplus \ker(P)$, so $S \in \GL_d(\F)$. Since $PS = [S_1, 0]$,
\[
S^* (I_{m,n} P) S = \begin{bmatrix} S_1^* \\ S_2^* \end{bmatrix} I_{m,n} \begin{bmatrix} S_1 & 0 \end{bmatrix} = \begin{bmatrix} S_1^* I_{m,n} S_1 & 0 \\ S_2^* I_{m,n} S_1 & 0 \end{bmatrix} = \begin{bmatrix} S_1^* I_{m,n} S_1 & 0 \\ 0 & 0 \end{bmatrix},
\]
where the last equality is by \eqref{eq:i-ortho}. By Sylvester's law of inertia again,
\[\mathsf{s}(I_{m,n} P) = \mathsf{s}(S^* I_{m,n} P S) = \mathsf{s}(S_1^* I_{m,n} S_1) + (0, 0, d-k).\]
Therefore, $\mathsf{s} (I_{m,n} \vert_{\im(P)}) = (p,q,0)$ and $\mathsf{s} (I_{m,n} \vert_{\ker(P)}) = (m-p,n-q,0)$.
Choose an $\omega_{m,n}$-orthonormal basis $\{u_1, \dots, u_p, v_1, \dots, v_q\}$ of $\im(P)$ such that $u_i^*I_{m,n}u_i = 1$, $i=1,\dots,p$, and $v_j^* I_{m,n}v_j=-1$, $j=1,\dots,q$. Similarly, choose an $\omega_{m,n}$-orthonormal basis $\{w_1, \dots, w_{m-p}, z_1, \dots, z_{n-q}\}$ of $\ker(P)$ such that $w_i^*I_{m,n}w_i = 1$, $i=1,\dots,m-p$, and $z_j^* I_{m,n}z_j=-1$, $j=1,\dots,n-q$. Let
\[
Q \coloneqq \bigl[ u_1 ,\dots, u_p , w_1, \dots, w_{m-p} , v_1, \dots, v_q , z_1, \dots, z_{n-q} \bigr].
\]
By construction, $Q^* I_{m,n} Q = I_{m,n}$, i.e., $Q \in \Iso_{m,n}(\F)$. Furthermore, 
\[
P [ u_1 ,\dots, u_p , w_1, \dots, w_{m-p} , v_1, \dots, v_q , z_1 ,\dots, z_{n-q} \bigr] = QD,
\]
so $P = QDQ^{-1}$. Applying \eqref{eq:equiv-model} completes the proof.
\end{proof}
Other choices of $(\lambda , \mu )$ also yield interesting embeddings. For example, if $(\lambda , \mu ) = (1,-1)$, then $\Gr^{1,-1}_{p,q}(\F^{m+n})$ is a set of involution matrices, an indefinite analog to the involution model in \cite{LLY2020}.

\section{Indefinite Stiefel manifold}

The standard Grassmannian of $k$-planes in $n$-space has a companion Stiefel manifold of orthonormal $k$-frames in $n$-space. It is the same for the indefinite Grassmannian. We define this in the most straightforward manner:
\begin{definition}[Indefinite Stiefel manifold]
The indefinite Stiefel manifold of $\omega$-orthonormal $k$-frames of signature $(p,q)$ is
\[
    \St_{p,q}(\F^{m+n}) \coloneqq \bigl\{Y\in\F^{d\times k}:Y^*I_{m,n}Y=I_{p,q}\bigr\}.
\]
Thus the columns of $Y$ form an ordered basis of a subspace on which
$\omega$ has Gram matrix $I_{p,q}$.
\end{definition}
An immediate observation is that this defines a smooth manifold with easily describable tangent spaces.
\begin{proposition}
The indefinite Stiefel manifold is a real-analytic embedded submanifold of
$\F^{d\times k}$ of dimension
\[
    \dim_{\R}\St_{p,q}(\F^{m+n})
    =\begin{cases}
        dk-k(k+1)/2 &\text{if }\F=\R,\\
        2dk-k^2 &\text{if }\F=\C,\\
        4dk-k(2k-1) &\text{if }\F=\H.
        \end{cases}
\]
For any $Y \in \St_{p,q}(\F^{m+n})$,
\[
    \mathbb{T}_Y\St_{p,q}(\F^{m+n})
    =\bigl\{Z\in\F^{d\times k}:Y^*I_{m,n}Z+Z^*I_{m,n}Y=0\bigr\}.
\]
\end{proposition}

\begin{proof}
Consider the real-analytic map $f:\F^{d\times k}\to \mathsf{H}^2(\F^k)$, $ Y \mapsto Y^*I_{m,n}Y$.
Then $(df)_Y(Z)=Z^*I_{m,n}Y+Y^*I_{m,n}Z$.
For $Y\in f^{-1}(I_{p,q})$ and $H\in\mathsf{H}^2(\F^k)$, taking $Z=\frac12YI_{p,q}H$ gives $Y^*I_{m,n}Z+Z^*I_{m,n}Y=H$. Therefore $(df)_Y$ is surjective. The result then follows from the regular level set theorem \cite[Theorem~9.9]{Tu}. 
\end{proof}

As is the case for the indefinite Grassmannian in Section~\ref{sec:many}, the indefinite Stiefel manifold has several characterizations. We highlight just one:
\begin{proposition}[Indefinite Stiefel manifold as homogeneous space]\label{prop:homoSt}
The group $\Iso_{m,n}(\F)$ acts transitively on $\St_{p,q}(\F^{m+n})$, giving an $\Iso_{m,n}(\F)$-equivariant diffeomorphism
\[
    \St_{p,q}(\F^{m+n})
    \cong
    \Iso_{m,n}(\F)/\Iso_{m-p,n-q}(\F).
\]
\end{proposition}
\begin{proof}
The action is matrix product, $\Iso_{m,n}(\F) \times  \St_{p,q}(\F^{m+n}) \to \St_{p,q}(\F^{m+n})$, $Y \mapsto QY$. Let
\[
 Y_0=\bigl[ e_1,\dots,e_p,e_{m+1},\dots,e_{m+q} \bigr]\in\St_{p,q}(\F^{m+n}).
\]
For $Y\in\St_{p,q}(\F^{m+n})$, there exists an $\omega$-isometry  $\varphi: \im(Y_0) \to \im(Y)$ such that the  ordered columns of $Y_0$ are mapped to the ordered columns of $Y$. By Witt's extension theorem (classical for $\R$ and $\C$ \cite[Theorem~3.9]{Artin1957} with $\H$ treated in \cite[Theorem~5.1]{H-Witt}), $\varphi$ extends to an $\omega$-isometry $Q$, and thus the action is transitive. On the other hand, let $\mathsf{H}$ be the stabilizer of $Y_0$. Then $Q \in \Iso_{m,n}(\F)$ is in $\mathsf{H}$ if and only if $Qe_i = e_i$, $i = 1,\dots, p$ and $i = m+1,\dots,m+q$. Hence $\mathsf{H} \cong \Iso_{m-p,n-q}(\F)$.
\end{proof}

The indefinite Stiefel manifold and Grassmannian are naturally tied as the total and base spaces in a principal bundle:
\begin{proposition}[Homogeneous fibration]
The map
\[
    \pi:\St_{p,q}(\F^{m+n})\to \Gr_{p,q}(\F^{m+n}),
    \quad Y \mapsto \im(Y),
\]
is a principal $\Iso_{p,q}(\F)$-bundle. Equivalently, it is the homogeneous fibration
\begin{equation}\label{eq:fibration}
    \Iso_{p,q}(\F) \to
    \frac{\Iso_{m,n}(\F)}{\Iso_{m-p,n-q}(\F)} \to
    \frac{\Iso_{m,n}(\F)}{\Iso_{p,q}(\F)\times\Iso_{m-p,n-q}(\F)}.
\end{equation}
\end{proposition}
\begin{proof}
Right multiplication $\Iso_{p,q}(\F) \times \St_{p,q}(\F^{m+n}) \to \St_{p,q}(\F^{m+n})$, $(Q,Y) \mapsto YQ$ defines a group action  as $(YQ)^*I_{m,n}(YQ)=Q^*I_{p,q}Q=I_{p,q}$.
It follows that $\pi(YQ) = \pi(Y)$ for all $Q\in\Iso_{p,q}(\F)$ and the action is free. If $\im(Y) = \im(Y')$, then there exists a
unique $Q\in\operatorname{GL}_k(\F)$ with $Y'=YQ$.  Since the Gram matrices of $Y$ and $Y'$ are equal on $\im(Y) = \im(Y')$, we get $Q^*I_{p,q}Q=I_{p,q}$, i.e., $Q\in\Iso_{p,q}(\F)$.
For local triviality, the standard homogeneous fibration \eqref{eq:fibration} is a smooth locally trivial principal bundle.
\end{proof}

\section{Geometry of the indefinite Grassmannian}\label{sec:geo}

The indefinite Grassmannian has rich geometry, endowed with a natural pseudo-Riemannian metric $g$, which turns it into not just a pseudo-Riemannian manifold but a symmetric space as well as an Einstein manifold. Over $\C$, it has, in addition, a natural complex structure $\mathcal{J}$ that together with $g$ makes it a pseudo-K\"ahler manifold. By ``natural'' we mean that both $g$ and $\mathcal{J}$ are $\Iso_{m,n}(\F)$-invariant.

To show that the indefinite Grassmannian is an Einstein manifold, we need its Ricci curvature. We will give explicit expression for this as well as other common intrinsic curvatures --- Riemann, sectional, scalar --- using the characterization of an indefinite Grassmannian as an adjoint orbit of $\Iso_{m,n}(\F)$ in Proposition~\ref{prop:equi-orb}. We will also provide one important extrinsic curvature, the second fundamental form of the embedding in Propositions~\ref{prop:equi-orb} and \ref{prop:i-proj}.

The characterization of $\Gr_{p,q}(\F^{m+n})$  in Proposition~\ref{prop:equi-orb} is the most convenient form for our purposes in this section. With an approriate pseudo-Riemannian metric on $\Gr_{p,q}^{\lambda ,\mu }(\F^{m+n})$, the diffeomorphism  $\Gr_{p,q}(\F^{m+n}) \cong \Gr_{p,q}^{\lambda ,\mu }(\F^{m+n})$  in  \eqref{eq:equiv-model} is an isometry of pseudo-Riemannian manifolds. In the following, we remind readers of the $\omega$-adjoint $\dagger$ in Section~\ref{sec:ila}.

\begin{proposition}[Pseudo-Riemannian metric on the indefinite Grassmannian]\label{prop:pseudo} \;
\begin{enumerate}[\upshape (i)]
\item The tangent space at $\W \in \Gr_{p,q}(\F^{m+n})$ is $\mathbb{T}_{\W}\Gr_{p,q}(\F^{m+n}) = \Hom_{\F}(\W,\W^{\perp_\omega})$ and 
\[
g_{\W}(A,B)
\coloneqq
\Re\tr(A^\dagger B),
\quad
A,B\in\Hom_{\F}(\W,\W^{\perp_\omega}),
\]
defines an $\Iso_{m,n}(\F)$-invariant pseudo-Riemannian metric
on $\Gr_{p,q}(\F^{m+n})$ with signature
\begin{equation}\label{eq:sign}
\bigl(
\dim_{\R}\F\,[p(m-p)+q(n-q)],
\dim_{\R}\F\,[p(n-q)+q(m-p)]
\bigr).
\end{equation}

\item The tangent space at $X\in\Gr_{p,q}^{\lambda ,\mu }(\F^{m+n})$ is
\begin{equation}\label{eq:tangent}
    \mathbb{T}_X\Gr_{p,q}^{\lambda ,\mu }(\F^{m+n})
=
\bigl\{
    Z\in\F^{d\times d}:
    I_{m,n}^{p,q}Z=Z^* I_{m,n}^{p,q},\,
    XZ+ZX=(\lambda +\mu )Z
\bigr\}
\end{equation}
and the real bilinear form
\[
g_X(Z_1,Z_2)
\coloneqq
\frac{\Re\tr(Z_1Z_2)}{2(\lambda -\mu )^2}
\]
defines an $\Iso(\F; I_{m,n}^{p,q})$-invariant pseudo-Riemannian metric on
$\Gr_{p,q}^{\lambda ,\mu }(\F^{m+n})$ with signature \eqref{eq:sign}.

\item Let $P_{\W}^{\omega}$ denote the $\omega$-orthogonal projection onto $\W$ along $\W^{\perp_\omega}$. Then the diffeomorphism \eqref{eq:equiv-model} may be chosen to be
\[
\Phi_{\lambda,\mu}:
\Gr_{p,q}(\F^{m+n})
\longrightarrow
\Gr_{p,q}^{\lambda,\mu}(\F^{m+n}),
\quad
\W\longmapsto
\lambda P_{\W}^{\omega}
+\mu(I-P_{\W}^{\omega}),
\]
which is an $\Iso_{m,n}(\F)$-equivariant isometry.
\end{enumerate}
\end{proposition}

\begin{proof}
By \eqref{eq:equiv-model}, differentiating  $I_{m,n}^{p,q}X=X^*I_{m,n}^{p,q}$ and $(X-\lambda I)(X-\mu I)=0$ on a smooth curve through $X$ gives ``$\subseteq$'' in \eqref{eq:tangent}. The main effort is to show ``$\supseteq$'' in \eqref{eq:tangent}. By \eqref{eq:equiv-model}, for $X \in \Gr_{p,q}^{\lambda,\mu}(\F^{m+n})$ there exists $Q \in \Iso(\F; I_{m,n}^{p,q})$ such that $Q^{-1}XQ = \lambda I_k \oplus \mu I_{d-k} \eqqcolon X_0$. Let $Z$ be as in the right-hand side of \eqref{eq:tangent} and $Z_0\coloneqq Q^{-1}ZQ$.
Then
\[
I_{m,n}^{p,q} Z_0 = Z_0^* I_{m,n}^{p,q}, \quad (\lambda  I_k \oplus \mu  I_{d-k})
Z_0
+
Z_0
(\lambda  I_k \oplus \mu  I_{d-k})
=
(\lambda +\mu) Z_0.
\]
Write $Z_0=
\begin{bsmallmatrix}
    Z_{11}&Z_{12}\\
    Z_{21}&Z_{22}
\end{bsmallmatrix}$. Then the second equation above gives $Z_{11} = 0$ and $Z_{22} = 0$, as  $\lambda \ne\mu $. The first equation above gives $Z_{12}=Z_{21}^{\dagger}$. Therefore $Z_0=
\bigl[\begin{smallmatrix}
    0 &B^\dagger\\
    B &0
\end{smallmatrix} \bigr]$ for some $B\in\F^{(d-k)\times k}$. Let
\[ K_0 \coloneqq \frac{1}{\lambda -\mu }
\begin{bmatrix}
    0&-B^\dagger\\
    B&0
\end{bmatrix}.
\]
Then $K_0^*(I_{p,q}\oplus I_{m-p,n-q}) = -(I_{p,q}\oplus I_{m-p,n-q})K_0$ and $Z_0 =[K_0, \lambda  I_k \oplus \mu  I_{d-k}]$; and for $K \coloneqq Q K_0 Q^{-1}$ we have $Z=[K,X]$. Finally, differentiating the curve $ t\mapsto \exp(tK)X\exp(-tK)$ gives us $Z \in \mathbb{T}_X \Gr_{p,q}^{\lambda ,\mu }(\F^{m+n})$. 

Clearly, $g_X$ is a symmetric real bilinear form. For $Z_1$, $Z_2 \in \mathbb{T}_{X_0} \Gr_{p,q}^{\lambda ,\mu }(\F^{m+n})$, write
\[
    Z_i = \biggl[
    \begin{matrix}
        0&B_i^\dagger\\
        B_i&0
    \end{matrix}\biggr],
    \quad i=1,2.
\]
Then $g_{X_0}(Z_1,Z_2) = \Re\tr(B_1^\dagger B_2)/ (\lambda -\mu )^2$ as $\Re\tr(Z_1Z_2)= 2\Re\tr(B_1^\dagger B_2)$. Write $B=
\begin{bsmallmatrix}
    B_{11}&B_{12}\\
    B_{21}&B_{22}
\end{bsmallmatrix}$ with $B_{11}\in\F^{(m-p)\times p}$, $B_{12}\in\F^{(m-p)\times q}$, $B_{21}\in\F^{(n-q)\times p}$, $B_{22}\in\F^{(n-q)\times q}$. Then
\[
\begin{aligned}
    g_{X_0}(Z,Z)
    =
    \frac{1}{(\lambda -\mu )^2}
    \bigl(
        \lVert B_{11}\rVert_\fb ^2
        +\lVert B_{22}\rVert_\fb ^2
        -\lVert B_{12}\rVert_\fb ^2
        -\lVert B_{21}\rVert_\fb ^2
    \bigr).
\end{aligned}
\]
The restriction of $g_{X_0}$ to $\F^{(m-p)\times p}\oplus\F^{(n-q)\times q}$ is positive while its restriction to $\F^{(m-p)\times q}\oplus\F^{(n-q)\times p}$ is negative; so $g_{X_0}$ is nondegenerate with signature  \eqref{eq:sign}. Lastly, we check that $ g_X(Z_1,Z_2) = g_{UXU^{-1}} (UZ_1U^{-1},UZ_2U^{-1})$ for all $U\in\Iso(\F; I_{m,n}^{p,q})$.

Since $P_{U\W}^{\omega} = UP_{\W}^{\omega}U^{-1}$ for all $U \in \Iso(\F;I_{m,n}^{p,q})$, $\Phi_{\lambda,\mu}$ is $\Iso(\F;I_{m,n}^{p,q})$-equivariant. Routine calculation shows that it is a diffeomorphism with inverse $\Phi_{\lambda,\mu}^{-1}(X) = \ker(X-\lambda I)$. The map $
t \mapsto \exp\bigl(t \bigl[\begin{smallmatrix}
    0 &-A^\dagger \\ A &0
\end{smallmatrix} \bigr] \bigr)\W
$ defines a curve at $\W\in\Gr_{p,q}(\F^{m+n})$ with tangent vector $A \in \Hom_{\F}(\W,\W^{\perp_\omega})$. Consequently, by the equivariance of $\Phi_{\lambda,\mu}$,
\[
(d\Phi_{\lambda,\mu})_{\W}(A) = \biggl[\begin{bmatrix}
    0 &-A^\dagger \\ A &0
\end{bmatrix}, \Phi_{\lambda,\mu}(\W)\biggr] = (\lambda - \mu ) \begin{bmatrix}
    0 &A^\dagger \\ A &0
\end{bmatrix}.
\]
Thus, for $A,B\in\Hom_{\F}(\W,\W^{\perp_\omega})$,
\[
g_X\bigl(
(d\Phi_{\lambda,\mu})_{\W}(A),
(d\Phi_{\lambda,\mu})_{\W}(B)
\bigr)
=
\frac{1}
     {2(\lambda-\mu)^2}\Re\tr
     \biggl((\lambda - \mu ) \begin{bmatrix}
    0 &A^\dagger \\ A &0
\end{bmatrix}(\lambda - \mu ) \begin{bmatrix}
    0 &B^\dagger \\ B &0
\end{bmatrix}\biggr) = g_{\W}(A,B),
\]
as required.
\end{proof}

To avoid clutter, we identify $\Gr_{p,q}(\F^{m+n})$ with $\Gr_{p,q}^{\lambda ,\mu }(\F^{m+n})$ below, noting that all geometric properties hold irrespective of the model used but that the explicit expressions assume the model $\Gr_{p,q}^{\lambda ,\mu }(\F^{m+n})$ in \eqref{eq:equiv-model}. In the rest of this section, $\lambda, \mu \in \R$ are fixed, distinct, and arbitrary. One might also use the model in Proposition~\ref{prop:i-proj}, i.e., fix $(\lambda, \mu) = (1, 0)$, to eliminate these constants.

\begin{proposition}\label{prop:geodesics-curvature}
For $X\in \Gr_{p,q}(\F^{m+n})$ and $Z\in
\mathbb{T}_X\Gr_{p,q}(\F^{m+n})$, let $K_X(Z)\coloneqq [Z,X]/(\lambda -\mu )^2$. Then, with respect to the metric $g$ in Proposition~\ref{prop:pseudo},
\begin{enumerate}[\normalfont(i)]
    \item $\Gr_{p,q}(\F^{m+n})$ is a pseudo-Riemannian symmetric space;
    \item\label{it:geodesic} $K_X(Z)\in \mathfrak{iso}_{m,n}(\F; I_{m,n}^{p,q})$ and the unique geodesic with 
    $\gamma(0)=X$ and $\dot\gamma(0)=Z$ is 
\[
        \gamma_{X,Z}(t)=\exp\bigl(tK_X(Z)\bigr) X \exp\bigl(-tK_X(Z)\bigr);
\]
    \item $\Gr_{p,q}(\F^{m+n})$ is geodesically complete;
    \item\label{it:rie} the Riemann curvature  is
    \[
    \mathsf{Rie}_X(U,V,W,Z) =\frac{1}{2(\lambda-\mu)^4}\Re \tr([[U,V],W]Z)
    \]
    for all $U$, $V$, $W$, $Z \in
    \mathbb{T}_X\Gr_{p,q}(\F^{m+n})$.
\end{enumerate} 
\end{proposition}

\begin{proof}
We have $K_X(Z)\in \mathfrak{iso}(\F; I_{m,n}^{p,q})$ as
\begin{equation}\label{eq:commutator}
    [Z,X]^* I_{m,n}^{p,q} = I_{m,n}^{p,q}[X,Z].
\end{equation}
By \eqref{eq:equiv-model} and \eqref{eq:tangent},
\begin{equation}\label{eq:geo1}
    XZ + ZX = (\lambda  + \mu )Z \quad \text{and} \quad X^2 = (\lambda + \mu )X - \lambda \mu  I.
\end{equation}
It follows that
\begin{equation}\label{eq:geo2}
    [[Z,X],X]=ZX^2-2XZX+X^2Z=(\lambda -\mu )^2Z
\end{equation}
and hence $[K_X(Z),X]=Z$. Let $Q_X \coloneqq (2X-(\lambda +\mu )I)/(\lambda -\mu )$. Then, by \eqref{eq:geo1}, $Q_X^2 =I$. 
Moreover, by \eqref{eq:equiv-model}, $Q_X^*I_{m,n}^{p,q}=I_{m,n}^{p,q}Q_X$, and consequently, $Q_X\in\Iso(\F;I_{m,n}^{p,q})$. The map
$\sigma_X:
\Gr_{p,q}(\F^{m+n})
\to
\Gr_{p,q}(\F^{m+n})$, $Y \mapsto Q_XYQ_X$ is thus an isometry.  It is also an involution as $Q_X^2=I$. Since $Q_X$ is a polynomial in $X$, $\sigma_X(X)=Q_XXQ_X=X$. Furthermore, for
$Z\in \mathbb{T}_X\Gr_{p,q}(\F^{m+n})$, $Q_XZ+ZQ_X=0$, so $(d\sigma_X)_X(Z) = Q_XZQ_X =-Z$. Thus $\sigma_X$ is the geodesic symmetry at $X$, and
$\Gr_{p,q}(\F^{m+n})$ is a
pseudo-Riemannian symmetric space. For such a space, the canonical connection coincides with the Levi--Civita connection \cite[Theorem~2.1.2]{Pseudo}. The expressions for geodesics and curvature tensor thus follow from \cite[Corollary~1.4.13 and Proposition~1.4.14]{Pseudo}. In particular, since the expression in \ref{it:geodesic} is defined for all $t\in\R$, the
metric is geodesically complete. As a map  $\mathsf R_X: \mathbb{T}_X\Gr_{p,q}(\F^{m+n}) \times \mathbb{T}_X\Gr_{p,q}(\F^{m+n}) \to \End(\mathbb{T}_X\Gr_{p,q}(\F^{m+n}))$, the Riemann curvature is
\[
\mathsf R_X(U,V)W 
    = -\bigl[[K_X(U),K_X(V)],W\bigr] 
    = \frac{1}{(\lambda-\mu)^2}[[U,V],W],
\]
where $U$, $V$, $W\in    \mathbb{T}_X\Gr_{p,q}(\F^{m+n})$, from which we obtain the equivalent expression in \ref{it:rie}.
\end{proof}

\begin{corollary}[intrinsic curvatures]
Let $X \in \Gr_{p,q}(\F^{m+n})$ and $U$, $V \in \mathbb{T}_X\Gr_{p,q}(\F^{m+n})$.
\begin{enumerate}[\normalfont(i)]
    \item The sectional curvature $\kappa_X: \mathbb{T}_X\Gr_{p,q}(\F^{m+n}) \times \mathbb{T}_X\Gr_{p,q}(\F^{m+n}) \to \R$ is given by
\[
    \kappa_X(U,V) = 
-\frac{
    2\Re\tr([U,V]^2)}
{
    \Re\tr(U^2)
    \Re\tr(V^2)
    -
    \bigl(\Re\tr(UV)\bigr)^2,
}
\]
where $U$, $V$ span a nondegenerate two-plane.
    \item The Ricci curvature $\mathsf{Ric}_X: \mathbb{T}_X\Gr_{p,q}(\F^{m+n}) \times \mathbb{T}_X\Gr_{p,q}(\F^{m+n}) \to \R$ is given by
\[
    \mathsf{Ric}_X(U,V)=
     \frac{\dim_\R \F \cdot(d+2)-4}{2(\lambda-\mu)^2}
    \Re\tr(UV).
\]
    \item The scalar curvature $\mathsf{Sca}_X = \tr(\mathsf{Ric}_X)$ is given by 
\[
    \mathsf{Sca}_X
    =
    \dim_\R \F \cdot k(d-k)\bigl(\dim_\R \F \cdot(d+2)-4\bigr).
\]
    \item The traceless Ricci curvature $\mathsf{Z}_X: \mathbb{T}_X\Gr_{p,q}(\F^{m+n}) \times \mathbb{T}_X\Gr_{p,q}(\F^{m+n}) \to \R$ is identically zero for all $X$, i.e., $\Gr_{p,q}(\F^{m+n})$ is an Einstein manifold.
\end{enumerate}
\end{corollary}
\begin{proof}
By Proposition~\ref{prop:geodesics-curvature},
\[
\mathsf{Rie}_X(U,V,V,U)=
\frac{1}{(\lambda-\mu)^4}
\Re\tr([U,V]VU)=
-\frac{1}{2(\lambda-\mu)^4}
\Re\tr([U,V]^2),
\]
giving the required expression for sectional curvature. By $\Iso(\F; I_{m,n}^{p,q})$-invariance, it suffices to calculate Ricci curvature at $X_0 = \diag(\lambda I_k, \mu I_{d-k})$. Let $X_{ijl} \coloneqq (\lambda - \mu)Z_{ \mathsf{i}_l E_{ij}}$ with $E_{ij} \in \F^{(d-k) \times k}$ the $(i,j)$th  elementary matrix and $\mathsf{i}_l$ the $l$th unit of $\F$ over $\R$ for $l =1,\dots, \dim_\R \F$. Direct calculation gives
\[
g_{X_0}(X_{ijl}, X_{i'j'l'}) = (I_{p,q})_{jj} (I_{m-p,n-q})_{ii} \delta_{i,i'} \delta_{j,j'} \delta_{l,l'}.
\]
Since for all $z \in \F$, 
\[
\mathsf{i}_{1} \overline{z} \mathsf{i}_{1} + \mathsf{i}_2 \overline{z} \mathsf{i}_2 + \dots + \mathsf{i}_{\dim_\R \F} \overline{z} \mathsf{i}_{\dim_\R \F} = (2 - \dim_\R \F) z,
\]
we get
\begin{align*}
    \sum_{i=1}^{d-k}\sum_{j=1}^k \sum_{l=1}^{\dim_\R\F} (I_{p,q})_{jj} (I_{m-p,n-q})_{ii} X_{ijl}^2 &= \dim_\R\F(\lambda-\mu)^2
\begin{bmatrix}
    (d-k)I_k&0\\
    0&kI_{d-k}
\end{bmatrix},\\
\sum_{i=1}^{d-k}\sum_{j=1}^k \sum_{l=1}^{\dim_\R\F} (I_{p,q})_{jj} (I_{m-p,n-q})_{ii} X_{ijl}VX_{ijl} &= (2- \dim_\R \F)(\lambda - \mu)^2 V.
\end{align*}
Therefore,
\[
    \mathsf{Ric}_{X_0}(U,V)
    =
    \sum_{i=1}^{d-k}\sum_{j=1}^k \sum_{l=1}^{\dim_\R\F} 
    (I_{p,q})_{jj} (I_{m-p,n-q})_{ii}
    \mathsf{Rie}_{X_0}
    (X_{ijl},U,V,X_{ijl})
\]
which simplifies to the required expression. By \eqref{eq:dim}, taking the pseudo-Riemannian trace gives the scalar curvature, from which we get $\mathsf{Z}_X = \mathsf{Ric}_X - \frac{1}{\dim_\R\F \cdot k(d-k)}\mathsf{Sca}_X \cdot g_X = 0$.
\end{proof}

Aside from these intrinsic curvatures, there is one extrinsic curvature that deserves recording.
Let $\mathsf{H}^2(\F^d;I_{m,n}^{p,q})$ be given the constant nondegenerate bilinear form $\overline g(A,B) \coloneqq \Re\tr(AB)/2(\lambda-\mu)^2$.

\begin{theorem}[Second fundamental form]
As a pseudo-Riemannian submanifold of $\mathsf{H}^2(\F^d;I_{m,n}^{p,q})$, the second fundamental form $\mathrm{II}_X: \mathbb{T}_X\Gr_{p,q}(\F^{m+n}) \times \mathbb{T}_X\Gr_{p,q}(\F^{m+n}) \to \mathbb{N}_X\Gr_{p,q}(\F^{m+n})$ is
\[
    \mathrm{II}_X(U,V)
    =
    \frac{(\lambda+\mu)I-2X}{(\lambda-\mu)^2}(UV+VU)
\]
for $X \in \Gr_{p,q}(\F^{m+n})$ and $U$, $V \in \mathbb{T}_X\Gr_{p,q}(\F^{m+n})$.
\end{theorem}

\begin{proof}
By Proposition~\ref{prop:equi-orb} and \eqref{eq:equiv-model}, it suffices to calculate at $X_0 = \diag(\lambda I_k, \mu I_{d-k})$. By Proposition~\ref{prop:pseudo},
\begin{equation}\label{eq:X0}
    \mathbb{T}_{X_0}
    \Gr_{p,q}(\F^{m+n})
    =
    \bigl\{
        Z_A:
        A\in\F^{(d-k)\times k}
    \bigr\},
    \qquad
    Z_A
    \coloneqq
    \begin{bmatrix}
        0&A^\dagger\\
        A&0
    \end{bmatrix}.
\end{equation}
Let $\Phi(A)
\coloneqq
\exp(K_{X_0}(Z_A)) X_0 \exp(-K_{X_0}(Z_A))$ with $K_{X_0}(Z_A)$ as in Proposition~\ref{prop:geodesics-curvature}. Since $K_{X_0}(Z_A) \in\mathfrak{iso}_{m,n}(\F)$, we have $\Phi(\F^{(d-k)\times k}) \subseteq \Gr_{p,q}(\F^{m+n})$. By \eqref{eq:geo2}, $(d\Phi)_0(A) = [K_{X_0}(Z_A), X_0] = Z_A$. So $\Phi$ parametrizes a neighborhood of $X_0$ when restricted to a neighborhood of zero.

Let $V = Z_B$ for $B \in \F^{(d-k) \times k}$ and extend it to a smooth local vector field around $X_0$ via
\[
\widetilde V\bigl(\Phi(C)\bigr)
\coloneqq
\exp(K_{X_0}(Z_C))V\exp(-K_{X_0}(Z_C)) \in \mathbb{T}_{\Phi(C)}\Gr_{p,q}(\F^{m+n}).
\]  
Since
\[\frac{d}{dt}\bigg \vert_{t=0}\Phi(tA)
=
[K_{X_0}(Z_A),X_0]
=
U,\]
for all $A \in \F^{(d-k) \times k}$, we get
\[
\frac{d}{dt}\bigg \vert_{t=0}
\exp(tK_{X_0}(Z_A))V\exp(-tK_{X_0}(Z_A)) = -\frac{1}{\lambda-\mu}
Q_{X_0}(Z_AZ_B+Z_BZ_A) \in \mathbb{N}_{X_0}\operatorname{Gr}_{p,q}(\F^{m+n}),
\]
with $Q_X$ as in the proof of Proposition~\ref{prop:geodesics-curvature}.
Therefore, $\mathrm{II}_{X_0}(Z_A,Z_B)
= -\frac{1}{\lambda-\mu}
Q_{X_0}(Z_AZ_B+Z_BZ_A)$, which extends to all $X$,$U$,$V$ by virtue of $\Iso_{m,n}(\F)$-equivariance.
\end{proof}
The second fundamental form for other embeddings may be similarly calculated. For example, for the model $\Gr_{p,q}^\pi(\F^{m+n})$ in Proposition~\ref{prop:i-proj}, it is given by $\mathrm{II}_X(U,V)    =    (I-2X)(UV+VU)$.

The special case $\F = \C$ has one additional noteworthy property:
\begin{proposition}[complex structure]\label{prop:pseudo-kahler}
For $X\in \Gr_{p,q}(\C^{m+n})$ and $Z\in
\mathbb{T}_X\Gr_{p,q}(\C^{m+n})$, let
\begin{equation}\label{eq:complex-structure-matrix}
\mathcal{J}_XZ
\coloneqq 
\frac{i}{\lambda -\mu }[Z,X].
\end{equation}
Then  $\mathcal{J}$ defines an $\Iso_{m,n}(\C)$-invariant complex structure on $\Gr_{p,q}(\C^{m+n})$ compatible with $g$ in that
\[
g_X(\mathcal{J}_XU,\mathcal{J}_XV) = g_X(U,V)
\]
for all $U$, $V \in \mathbb{T}_X\Gr_{p,q}(\C^{m+n})$. Moreover, $\Gr_{p,q}(\C^{m+n})$ is a pseudo-K\"ahler manifold with respect to $g$ and $\mathcal{J}$. In particular, $\Omega(U,V) \coloneqq g(\mathcal{J}_X U, V)$ defines a symplectic form on $\Gr_{p,q}(\C^{m+n})$.
\end{proposition}

\begin{proof}
Let $X\in \Gr_{p,q}(\C^{m+n})$ and $Z\in
\mathbb{T}_X\Gr_{p,q}(\C^{m+n})$. By \eqref{eq:commutator}, $i[Z,X]$ is $I_{m,n}^{p,q}$-self-adjoint. Since $X[Z,X]+[Z,X]X = [Z,X^2]= (\lambda +\mu )[Z,X]$, we obtain $\mathcal{J}_X Z \in \mathbb{T}_X\Gr_{p,q}(\C^{m+n})$. By \eqref{eq:geo2}, $\mathcal{J}_X^2Z = -Z$, so $\mathcal{J}$ is an almost complex structure. The $\Iso(\C; I_{m,n}^{p,q})$-invariance follows from
\[
\mathcal{J}_{AXA^{-1}}(AZA^{-1})= \frac{i}{\lambda -\mu }[AZA^{-1},AXA^{-1}] =A(\mathcal{J}_XZ)A^{-1},
\]
for all $A\in\Iso_{m,n}(\C)$. Since $-\Re\tr([U,X][V,X])= \Re\tr([[U,X],X]V)=(\lambda -\mu )^2\Re\tr(UV)$ for all $U$, $V \in \mathbb{T}_X\Gr_{p,q}(\C^{m+n})$, we get
$g_X(\mathcal{J}_XU,\mathcal{J}_XV) =\Re\tr(UV)/2(\lambda -\mu )^2$ as required.
By \cite[Theorem~2.1.2, Proposition~1.4.15]{Pseudo}, $\nabla\mathcal{J}=0$, so $(g,\mathcal{J})$ is pseudo-K\"ahler. Consequently, the associated fundamental form $\Omega(U,V)= g(\mathcal{J}U,V)$ is nondegenerate, as both $g$ and $\mathcal{J}$ are nondegenerate. Moreover, since $\nabla g=0$ and $\nabla\mathcal{J}=0$, we have $\nabla\Omega=0$, and thus $d\Omega=0$. Hence $\Omega$ is a symplectic form.
\end{proof}

\section{Topology of the indefinite Grassmannian}\label{sec:top}

As a real manifold, the indefinite Grassmannian deformation retracts to a product of two Grassmannians determined by the pairs $(p,m)$ and $(q,n)$. Recall that results stated for $\F$ holds for all of $\R$, $\C$, $\H$.
\begin{theorem}[Strong deformation retract]\label{thm:sdr}
The inclusion 
\[
\Gr_{p}(\F^{m})\times \Gr_{q}(\F^{n}) \hookrightarrow \Gr_{p,q}(\F^{m+n}), \quad (\W_1,\W_2) \mapsto \W_1 \oplus \W_2
\]
induces a strong deformation retract of $\Gr_{p,q}(\F^{m+n})$ onto $\Gr_{p}(\F^{m})\times \Gr_{q}(\F^{n})$.
\end{theorem}
\begin{proof}
Consider the subspaces
\begin{alignat*}{4}
  \W_1 &= \operatorname{span}\{e_1,\dots,e_p\}, \quad &\W_2 &= \operatorname{span} \{e_{m+1},\dots, e_{m+q}\}, & &\\
  \W_3 &= \operatorname{span}\{e_{p+1},\dots,e_m\}, \quad &\W_4 &= \operatorname{span}\{e_{m+q+1},\dots,e_{d}\}, \quad &\W_5 &= \W_1 \oplus \W_2
\end{alignat*}
under the action of $\Iso_{m,n}(\F)$. 
By \cite[Theorem~3.1]{Mostow} and \cite[Chapter~X, Section~2]{Helgason2001}, it suffices to verify that 
\begin{equation}\label{eq:tmp1}
\adjustbox{max width=0.9\textwidth}{$\dfrac{\Iso_{m,n}(\F) \cap \Iso_{m+n}(\F)}{\Iso_{m,n}(\F) \cap \Iso_{m+n}(\F) \cap \stab(\W_5)}= \dfrac{\diag(\Iso_m(\F), \Iso_n(\F))}{\diag(\Iso_m(\F) , \Iso_n(\F)) \cap \stab(\W_5) }\cong \Gr_{p}(\F^m) \times \Gr_{q}(\F^n).$}
\end{equation}
Let $Q \in \Iso_m(\F)$ and $Q' \in \Iso_n(\F)$ be such that $Q \W_1 = \W_1$ and $Q' \W_2 = \W_2$. Since $Q$ acts unitarily on $\W_1 \oplus \W_3$, $Q = \diag(Q_1,Q_2)$ where $Q_1 \in \Iso_p(\F)$ and $Q_2 \in \Iso_{m-p}(\F)$. Similarly, $Q' = \diag(Q_1',Q_2')$ where $Q_1' \in \Iso_q(\F)$ and $Q_2' \in \Iso_{n-q}(\F)$. Therefore,
\[
\diag(\Iso_m(\F), \Iso_n(\F)) \cap \stab(\W_5) = \diag \bigl(\Iso_p(\F), \Iso_{m-p}(\F), \Iso_q(\F),\Iso_{n-q}(\F)\bigr).
\]
Consequently, the space in the middle of \eqref{eq:tmp1} is isomorphic to
$\bigl[\Iso_m(\F) \times \Iso_n(\F)\bigr] / \bigl[ \Iso_p(\F)\times \Iso_{m-p}(\F)\times \Iso_q(\F)\times \Iso_{n-q}(\F) \bigr]$, which is in turn isomorphic to $\Gr_{p}(\F^m) \times \Gr_{q}(\F^n)$.
\end{proof}

By Theorem~\ref{thm:sdr}, standard topological invariants can be obtained via homotopy equivalence. We will record a few below (homotopy, cohomology, characteristic class) for convenient future reference.
\begin{corollary}[Homotopy groups]\label{cor:hom-grp}
For any $l\geq 1$,
\[
\pi_l\bigl(\Gr_{p,q}(\F^{m+n})\bigr)
\cong
\pi_l\bigl(\Gr_p(\F^m)\bigr)
\times
\pi_l\bigl(\Gr_q(\F^n)\bigr).
\]
Let
\[
s \coloneqq 
\begin{cases}
\dim_\R\F\bigl(\min\{q, n-q\} + 1\bigr) - 2 & \text{if } (p=0 \text{ or } p=m) \text{ and } (0 < q < n),\\ 
\dim_\R\F\bigl(\min\{p, m-p\} + 1\bigr) - 2 & \text{if } (0 < p < m) \text{ and } (q=0 \text{ or } q=n),\\ 
\dim_\R\F\bigl(\min\{p, m-p, q, n-q\} + 1\bigr) - 2 & \text{if } (0 < p < m) \text{ and } (0 < q < n),\\
 +\infty & \text{if } (p=0 \text{ or } p=m) \text{ and } (q=0 \text{ or } q=n).
\end{cases}
\]
Then for $l =1,\dots, s$,
\[
\pi_l\bigl(\Gr_{p,q}(\F^{m+n})\bigr)
\cong
\bigl(\Pi_l^{\F}\bigr)^{\boldsymbol{1}_{\{0<p<m\}}
+\boldsymbol{1}_{\{0<q<n\}}}.
\]
\end{corollary}
\begin{proof}
The inclusion $\iota\colon
\Gr_p(\F^m)\times\Gr_q(\F^n)
\hookrightarrow
\Gr_{p,q}(\F^{m+n})$ in Theorem~\ref{thm:sdr} is a
homotopy equivalence and induces isomorphisms on all homotopy groups. We then use the product formula for homotopy groups \cite[Proposition~4.2]{Hatcher}.

Computing general homotopy groups of the standard Grassmannian is well-known to be difficult, but Bott periodicity gives partial results for special cases \cite[Chapter~IV, Theorem~6.2, Table~4.1]{MimuraToda}. Let
\[
\Pi_l^{\R}= 
\begin{cases}
    \Z & l\equiv 0,4\mod 8,\\
    \Z_2 & l\equiv 1,2\mod 8,\\
    0 & \text{otherwise},
\end{cases}\quad
\Pi_l^{\C}
=
\begin{cases}
    \Z & l\equiv 0\mod 2,\\
    0 & l\equiv 1\mod 2,
\end{cases} \quad
\Pi_l^\H=
\begin{cases}
\Z & l\equiv0,4\mod8,\\
\Z_2 & l\equiv5,6\mod8,\\
0 & \text{otherwise}.
\end{cases}
\]
If $p=1,\dots,m-1$, then for $l = 1,\dots, \dim_{\mathbb{R}}\mathbb{F}\cdot \bigl(\min\{p,m-p\}+1\bigr)-2$, we have $\pi_l\bigl(\Gr_p(\mathbb{F}^m)\bigr) \cong \Pi_l^{\mathbb{F}}$  \cite[Chapter~II]{MimuraToda}. If $p=0$ or $p=m$, the Grassmannian is simply a point.
\end{proof}

Next we state its cohomology, again as a public service for easy future references.
\begin{corollary}[Cohomology ring]\label{cor:cohom}
For every coefficient field $\Bbbk$, there is an isomorphism of
graded $\Bbbk$-algebras
\begin{equation}\label{eq:cohom1}
H^*\bigl(\Gr_{p,q}(\F^{m+n});\Bbbk\bigr)
\cong
H^*\bigl(\Gr_p(\F^m);\Bbbk\bigr)
\otimes_{\Bbbk}
H^*\bigl(\Gr_q(\F^n);\Bbbk\bigr),
\end{equation}
where the tensor product is endowed with its usual graded-algebra structure.  Let $\mathcal{T}$ be the tautological bundle over $\Gr_p(\R^m)$ and consider the quotient bundle $\mathcal{Q} = (\Gr_p(\R^m) \times \R^m)/\mathcal{T}$. Let $w$ be the Stiefel--Whitney class. Then 
$H^*\bigl(\Gr_p(\R^m);\Z_2\bigr) $ is given by
\begin{equation}\label{eq:cohom}
\dfrac{\Z_2[w_1(\mathcal{T}),\dots,w_p(\mathcal{T}),w_1(\mathcal{Q}),\dots,w_{m-p}(\mathcal{Q})]}{\langle[t^j]\bigl((1+w_1(\mathcal{T})t+\dots + w_p(\mathcal{T})t^p)(1+w_1(\mathcal{Q})t+\dots+ w_{m-p}(\mathcal{Q})t^{m-p})-1\bigr): j=1,\dots, m\rangle}.
\end{equation}
\end{corollary}
\begin{proof}
Pulling back along the inclusion $\iota$ in Theorem~\ref{thm:sdr}  induces an isomorphism of graded
$\Bbbk$-algebras
\[
\iota^*\colon
H^*\bigl(\Gr_{p,q}(\F^{m+n});\Bbbk\bigr)
\xrightarrow{\sim}
H^*\bigl(
\Gr_p(\F^m)\times\Gr_q(\F^n);
\Bbbk
\bigr),
\]
from which we obtain \eqref{eq:cohom1} with K\"unneth theorem
\cite[Theorem~3.15]{Hatcher}. See \cite{Borel1953} for the cohomology ring of $\Gr_p(\R^m)$ over $\Z_2$.
\end{proof}
Over $\C$ and $\H$, there are of course analogs of \eqref{eq:cohom} in terms of Chern and symplectic Pontryagin classes respectively; unfortunately, these have coefficient ring $\Z$ and \eqref{eq:cohom1} does not apply. We leave these and other more involved calculations for $H^*\bigl(\Gr_{p,q}(\F^{m+n});\Z\bigr)$ to interested readers. Instead here we will say more about the characteristic classes of the indefinite Grassmannian.
\begin{corollary}[characteristic classes]
Let $\iota$ and $\mathcal{T}$ be as in Theorem~\ref{thm:sdr} and Corollary~\ref{cor:cohom} respectively. Let $\mathcal{T}_p$ be the tautological bundle over $\Gr_p(\F^m)$. Then, for all $l=0,1,2,\dots$,
\[
\iota^* w_l(\mathcal{T}) = \sum_{i+j = l} w_i(\mathcal{T}_p)w_j(\mathcal{T}_q),
\]
where $w_i$ denotes Stiefel--Whitney, Chern, or symplectic Pontryagin class when $\F = \R$, $\C$, or $\H$ respectively.
\end{corollary}
\begin{proof}
The result follows from the Whintney product formula for Stiefel--Whitney, Chern, and symplectic Pontrjagin classes \cite[Section~9.7]{BH1958}.
\end{proof}

For a generalized Grassmannian in the sense of \cite{flag}, the orbits of a Borel subgroup give a stratification into Schubert cells. We do not have a stratification of this nature as the indefinite Grassmannian is not a generalized Grassmannian, as we mentioned in Section~\ref{sec:intro}. Nevertheless, as we will see next, the indefinite Grassmannians and the standard Grassmannian are related in a different stratification entirely distinct from that given by Schubert cells.

\section{Stratification of the Grassmannian by indefinite Grassmannians}\label{sec:strat}

By the same Witt's extension argument in Proposition~\ref{prop:hom}, the natural action of $\Iso_{m,n}(\F)$ on $\Gr_k(\F^d)$ yields a partition  into finitely many orbits:
\begin{equation}\label{eq:partition-grass}
    \Gr_k(\F^d) \;=\; \bigsqcup_{\substack{p+q+r=k \\ p+r \le m, \, q+r \le n}} \mathcal{O}_{(p,q,r)},
    \qquad
    \mathcal{O}_{(p,q,r)} \coloneqq  \bigl\{\W \in \Gr_k(\F^d) : \mathsf{s}(\omega_{m,n}\vert_\W) = (p,q,r) \bigr\}.
\end{equation}
In particular, Proposition~\ref{prop:hom} shows that
\begin{equation}\label{eq:partition-grass1}
    \mathcal{O}_{(p,q,0)} = \Gr_{p,q}(\F^{m+n}), \qquad p+q=k.
\end{equation}
The goal of this section is to establish that \eqref{eq:partition-grass} is a Whitney~(B) stratification of $\Gr_k(\F^d)$ and that the indefinite Grassmannians in  \eqref{eq:partition-grass1} are exactly the highest dimensional strata. We emphasize that this stratification is different from the standard stratification via group actions \cite[Chapter~4]{Pflaum} where the space is stratified by \emph{orbit types}, noting that the strata in \eqref{eq:partition-grass} are de facto \emph{orbits}. Moreover, the orbit-type stratification requires the group action to be proper, which is not the case here, as our stabilizer is noncompact  by Proposition~\ref{prop:hom} \cite[Theorem~4.2.4]{Pflaum}.

Let $\mathcal{M}$ be a smooth manifold and $Z \subseteq \mathcal{M}$ a locally closed subset. By a \emph{stratification}  of $Z$, we mean a locally finite partition $Z=\bigsqcup_{a\in A} S_a$ into nonempty, locally closed smooth submanifolds of $\mathcal{M}$, called \emph{strata}, satisfying the \emph{frontier condition}: For all $a,b\in A$, $ S_a\cap\overline{S_b}\ne\varnothing$ if and only if $S_a\subseteq\overline{S_b}$. Let $S_a\cap\overline{S_b}\ne\varnothing$ and $y\in S_a$. We say $(S_a, S_b)$ satisfies \emph{Whitney's Condition~(B)} at $y$ if the following holds: Let $(x_i)\subseteq S_b$ and $(y_i)\subseteq S_a$ with $x_i\to y$ and $y_i\to y$. If $\ell_i=\overline{x_i y_i}$ converge to a line $\ell$ in some local coordinate and $\mathbb{T}_{x_i}S_b$ converge to a subspace $\mathbb{U}$, then $\ell\subseteq\mathbb{U}$. A stratification is a \emph{Whitney (B) stratification} if every incident pair of strata satisfies Condition~(B) at every point of the smaller stratum. This is a stricter notion of stratification than many others: It is well-known that a Whitney~(B) stratification is a Whitney~(A) stratification and a Thom--Mather stratification \cite{Mather}. 

Let $\W$ be a $d$-dimensional vector space over $\F$ equipped with the Euclidean form $\langle\,\cdot,\cdot\, \rangle$. We define the following sets of Hermitian operators on $\W$:
\[
\mathsf{H}^2(\W) \coloneqq \{H\in\End(\W):H^*=H\}, \quad
\mathsf{H}^2(\W)^{(p,q,r)} \coloneqq
\{H\in \mathsf{H}^2(\W):\mathsf{s}(H)=(p,q,r)\}.
\]
It is well-known that the partition
\[
\mathsf{H}^2(\W)=\bigsqcup_{i=0}^d \,\{H\in \mathsf{H}^2(\W):\dim_{\F}\ker H=i\}
\]
is a Whitney stratification \cite{Arnold}; we will define a different, finer stratification by $\mathsf{H}^2(\W)^{(p,q,r)}$.

\begin{lemma}\label{lem:embed}
Let $\W$ be a $d$-dimensional vector space over $\F$. If $p$, $q$, $r\ge0$ such that $p+q+r=d$, then $\mathsf{H}^2(\W)^{(p,q,r)}$ is a  real-analytic embedded submanifold of $\mathsf{H}^2(\W)$, with 
\[
\codim_{\R} \mathsf{H}^2(\W)^{(p,q,r)}=
\begin{cases}
    r(r+1)/2 &\text{if } \F = \R,\\
    r^2 &\text{if } \F = \C,\\
    r(2r-1) &\text{if } \F = \H.
\end{cases}
\]
Moreover, the partition
\begin{equation}\label{eq:herm-stra}
\mathsf{H}^2(\W) =\bigsqcup_{p+q+r=d} \mathsf{H}^2(\W)^{(p,q,r)}
\end{equation}
is a Whitney~(B) stratification.
\end{lemma}
\begin{proof}
Let $H_0\in \mathsf{H}^2(\W)^{(p,q,r)}$ and $\K \coloneqq \ker H_0$. Write $\K^\perp$ for its orthogonal complement with respect to the Euclidean form. Partition $H \in \mathsf{H}^2(\W)^{(p,q,r)}$ into
\[
    H=\begin{bmatrix}A&B\\ B^*&C\end{bmatrix}, \quad A \in \End(\K^\perp),\, B \in \Hom(\K, \K^\perp),\, C \in \End(\K).
\]
In particular, since $\langle H_0x,y\rangle=\langle x,H_0y\rangle=0$ for all $x \in \K^\perp$ and $y \in \K$, we have
\[
H_0=\begin{bmatrix}A_0& 0\\ 0 &0\end{bmatrix}, \quad A_0 \in\GL(\K^\perp),\quad \mathsf{s}(A_0) = \mathsf{s}(H_0) = (p,q,0).
\]
Consider the open neighborhood  of $ H_0$,
\[
\mathcal{I} \coloneqq \biggl\{\begin{bmatrix}A&B\\ B^*&C\end{bmatrix}\in \mathsf{H}^2(\W):\mathsf{s}(A)=\mathsf{s}(A_0)=(p,q,0)\biggr\}.
\]
For $H \in \mathcal{I}$, denote its Schur complement by $S(H) \coloneqq C - B^*A^{-1}B \in \mathsf{H}^2(\K)$. It is clear that
\[
\Theta:\mathcal{I} \to \mathsf{H}^2(\K^\perp)^{(p,q,0)}\times\Hom(\K,\K^\perp)\times\mathsf{H}^2(\K), \quad H \mapsto (A,B, S(H))
\]
is a real-analytic diffeomorphism onto its image with inverse
\[
\Theta^{-1}(A,B,S)=\begin{bmatrix}A&B\\B^*&S+B^*A^{-1}B\end{bmatrix}.
\]
In particular, $\Theta(H_0) = (A_0,0,0)$. Since
\[
\begin{bmatrix}I_{\K^\perp}&-A^{-1}B\\0&I_\K\end{bmatrix}^*
H
\begin{bmatrix}I_{\K^\perp}&-A^{-1}B\\0&I_\K\end{bmatrix}
=\begin{bmatrix}A&0\\0&S(H)\end{bmatrix},
\]
we get
\begin{equation}\label{eq:add}
    \mathsf{s}(H)=(p,q,0)+\mathsf{s}\bigl(S(H)\bigr)\quad\text{for all }H\in\mathcal{I}.
    \end{equation}
Therefore, for open sets $\mathcal{U},\mathcal{V}$ such that
\[
(A_0,0)\in\mathcal{U}\subset
\mathsf{H}^2(\K^\perp)^{(p,q,0)}\times\Hom(\K,\K^\perp)
\quad \text{and} \quad 0\in\mathcal{V}\subset\mathsf{H}^2(\K),
\]
we have  
\begin{equation}\label{eq:local-theta}
\Theta\bigl(\mathsf{H}^2(\W)^{(a+p,b+q,c)}\cap\Theta^{-1}(\mathcal{U}\times\mathcal{V})\bigr)
=\mathcal{U}\times\bigl(\mathcal{V}\cap\mathsf{H}^2(\K)^{(a,b,c)}\bigr),
\quad \text{for all } a+b+c = r.
\end{equation}
Choosing $(a,b,c)=(0,0,r)$ gives $\Theta\bigl(\mathsf{H}^2(\W)^{(p,q,r)}\cap\Theta^{-1}(\mathcal{U}\times\mathcal{V})\bigr)
=\mathcal{U}\times\{0\}$. Since $H_0$ is arbitrary, $\mathsf{H}^2(\W)^{(p,q,r)}$ is a real-analytic embedded submanifold of $\mathsf{H}^2(\W)$ and
\[
\codim_{\R}\mathsf{H}^2(
\W)^{(p,q,r)}
=\dim_{\R}\mathsf{H}^2(\K)=
\begin{cases}
    r(r+1)/2 &\text{if } \F = \R,\\
    r^2 &\text{if } \F = \C,\\
    r(2r-1) &\text{if } \F = \H.
\end{cases} 
\]
Next we check the two conditions required for a Whitney~(B) stratification.

For the frontier condition, let $\mathsf{H}^2(\W)^{(p',q',r')}$ be any stratum. We claim that
\begin{equation}
\label{eq:front-iff}
H_0 \in \overline{\mathsf{H}^2(\W)^{(p',q',r')}}  \quad \iff \quad p' \ge p \text{ and } q' \ge q.
\end{equation}
Suppose $H_0 \in \overline{\mathsf{H}^2(\W)^{(p',q',r')}}$. Choose $H_i \to H_0$ with $H_i \in \mathsf{H}^2(\W)^{(p',q',r')}$. So $H_i \in \mathcal{I}$ for all large $i$ as $\mathcal{I}$ is open. Thus \eqref{eq:add} gives
\[
(p',q',r') = \mathsf{s}(H_i) = (p,q,0)+\mathsf{s}\bigl(S(H_i)\bigr).
\]
As $\mathsf{s}(S(H_i))$ has all coordinates nonnegative, $p' \ge p$ and $q' \ge q$. Conversely, suppose $p' \ge p$ and $q' \ge q$. Since $p'+q'+r' = d = p+q+r$, we have $r' = r-(p'-p)-(q'-q) \ge 0$ and
\[
(p'-p) + (q'-q) + r' = r = \dim_{\F} \K.
\]
So there exists $T \in \mathsf{H}^2(\K)$ with $\mathsf{s}(T) = (p'-p,q'-q,r')$. For $t > 0$, set $H_t \coloneqq \Theta^{-1}(A_0,0,\,tT)$. Then $\Theta(H_t) = (A_0,0,tT) \to (A_0,0,0) = \Theta(H_0)$. So $H_t \to H_0$. Also, $\mathsf{s}(tT) = \mathsf{s}(T)$ and \eqref{eq:add} give
\[
\mathsf{s}(H_t) = (p,q,0) + (p'-p,q'-q,r') = (p',q',r').
\]
Thus $H_t \in \mathsf{H}^2(\W)^{(p',q',r')}$ and $H_0 \in \overline{\mathsf{H}^2(\W)^{(p',q',r')}}$.

Since the right-hand side of \eqref{eq:front-iff} does not depend on the choice of $H_0 \in \mathsf{H}^2(\W)^{(p,q,r)}$, either $\mathsf{H}^2(\W)^{(p,q,r)} \subseteq \overline{\mathsf{H}^2(\W)^{(p',q',r')}}$ or $\mathsf{H}^2(\W)^{(p,q,r)} \cap \overline{\mathsf{H}^2(\W)^{(p',q',r')}} = \varnothing$, which is the frontier condition.

It remains to check the Whitney~(B) condition. We need to verify that, for all nonnegative integers $a+b+c = r$, the incident pair $\bigl(\mathsf{H}^2(\W)^{(p,q,r)}, \mathsf{H}^2(\W)^{(p+a,q+b,c)} \bigr)$ satisfies Whitney's condition~(B) at $H_0$. Indeed, by \eqref{eq:local-theta}, we just need to verify this for the pair $\bigl(\mathcal{U}\times(\mathcal{V}\cap\mathsf{H}^2(\K)^{(a,b,c)}), \mathcal{U}\times\{0\}\bigr)$, as the condition is invariant under diffeomorphisms. Since $\{(1+t)S : t \ge 0\} \subseteq \mathsf{H}^2(\K)^{(a,b,c)}$ for all $S \in \mathsf{H}^2(\K)^{(a,b,c)}$,
\begin{equation}\label{eq:s-tan}
S = \frac{d}{dt}(1+t)S \Bigr\vert_{t=0} \in \mathbb{T}_S \mathsf{H}^2(\K)^{(a,b,c)}.
\end{equation}
Let $(A,B,0) \in \mathcal{U} \times \{0\}$. Let $\{(A_i,B_i,0)\}_{i=1}^\infty \subseteq \mathcal{U} \times \{0\}$ and $\{(C_i,D_i,S_i)\}_{i=1}^\infty \subseteq \mathcal{U} \times (\mathcal{V} \cap \mathsf{H}^2(\K)^{(a,b,c)})$ be two sequences approaching $(A,B,0)$. Suppose the line $\ell_i = \operatorname{span}_{\R}\{C_i - A_i, D_i - B_i, S_i \}$ converges to a line $\ell$ and $\mathbb{T}_{(C_i,D_i,S_i)}\bigl(\mathcal{U} \times (\mathcal{V} \cap \mathsf{H}^2(\K)^{(a,b,c)})\bigr)$ converges to a subspace $\mathbb{U}$. Since
\[
\mathbb{T}_{(C_i,D_i,S_i)}\bigl(\mathcal{U} \times (\mathcal{V} \cap \mathsf{H}^2(\K)^{(a,b,c)})\bigr) = \mathbb{T}_{(C_i,D_i)}\mathcal{U} \times \mathbb{T}_{S_i}\bigl(\mathcal{V} \cap \mathsf{H}^2(\K)^{(a,b,c)}\bigr),
\]
and $\mathbb{T}_{(C_i,D_i)}\mathcal{U} = \mathsf{H}^2(\K^\perp) \times \Hom(\K,\K^\perp)$ is the full space (as $\mathcal{U}$ is open), applying \eqref{eq:s-tan} gives $\ell_i = \operatorname{span}_{\R}\{C_i - A_i, D_i - B_i, S_i \} \subseteq \mathbb{T}_{(C_i,D_i,S_i)}\bigl(\mathcal{U} \times (\mathcal{V} \cap \mathsf{H}^2(\K)^{(a,b,c)})\bigr)$ for each $i$. Therefore $\ell \subseteq \mathbb{U}$ upon taking limits.
\end{proof}

We now prove the claimed stratification.
\begin{theorem}[Whitney stratification of $\Gr_k(\F^d)$ by $\Gr_{p,q}(\F^{m+n})$]\label{thm:whitney-grass}
For $\F=\R$, $\C$, or $\H$, the partition \eqref{eq:partition-grass} of $\Gr_k(\F^d)$ is a real-analytic Whitney~(B) stratification with
\[
\codim_\R \mathcal{O}_{(p,q,r)} =
\begin{cases}
    r(r+1)/2 & \text{if } \F = \R,\\
    r^2 & \text{if } \F = \C,\\
    r(2r-1) & \text{if } \F = \H.
\end{cases}
\]
\end{theorem}

\begin{proof}
Write $\omega\coloneqq \omega_{m,n}$. Let $\W \in \mathcal{O}_{(p,q,r)}$, so $\mathsf{s}(\omega\vert_{\W}) = (p,q,r)$. Let $\K \coloneqq \ker(\omega\vert_{\W})$, so $\dim_{\F}\K = r$. Let $\E \coloneqq \W \cap \K^{\perp}$, so $\omega \vert_{\E}$ is nondegenerate.  Define $\Phi \colon \Hom(\W, \W^{\perp}) \to \mathsf{H}^2(\W)$ by
\[
\langle x, \Phi(L)y \rangle = \omega\bigl((I+L)x, (I+L)y\bigr), \quad x,y \in \W,
\]
and write it in block form with respect to $\W = \E \oplus \K$,
\[
\Phi(L) = \begin{bmatrix} A(L) & B(L) \\ B(L)^* & C(L) \end{bmatrix},
\quad A(L) \in \mathsf{H}^2(\E),\ B(L) \in \Hom(\K, \E),\ C(L) \in \mathsf{H}^2(\K).
\]
Since $\langle x, \Phi(0)y\rangle = \omega(x,y)$ for $x,y \in \W$, we have $\mathsf{s}(\Phi(0)) = \mathsf{s}(\omega\vert_{\W}) = (p,q,r)$. As $\K = \ker(\omega\vert_{\W})$, this forces $B(0) = 0$, $C(0) = 0$, and $A(0) \in \GL(\E)$ with $\mathsf{s}(A(0)) = (p,q,0)$. By restricting to a sufficiently small neighborhood $\mathcal{I} \ni 0$, we may assume that $A(L) \in \GL(\E)$ with constant inertia $\mathsf{s}(A(L)) = (p,q,0)$ for all $L \in \mathcal{I}$.

The Schur complement of $\Phi(L)$ defines a map
\begin{equation}\label{eq:psi}
    \Psi \colon \mathcal{I} \to \mathsf{H}^2(\K), \quad L \mapsto C(L) - B(L)^* A(L)^{-1} B(L),
\end{equation}
so that
\begin{equation}\label{eq:shur-gr}
\begin{bmatrix} I_{\E} & -A(L)^{-1}B(L) \\ 0 & I_{\K} \end{bmatrix}^{ *}
\Phi(L)
\begin{bmatrix} I_{\E} & -A(L)^{-1}B(L) \\ 0 & I_{\K} \end{bmatrix}
= \begin{bmatrix} A(L) & 0 \\ 0 & \Psi(L) \end{bmatrix}.
\end{equation}
We claim that $\Psi$ is a submersion at $0 \in \Hom(\W, \W^{\perp})$. By definition,
\[
d\Psi = dC - (dB)^*A^{-1}B - B^*A^{-1}\,dB + B^*A^{-1}(dA)A^{-1}B.
\]
Since $B(0) = 0$, we have $d\Psi \vert_0 = dC\vert_0$. For $X\in\Hom(\W,\W^\perp)$ and $x,y\in\K$,
\begin{equation}\label{eq:diff}
\langle x,((d\Psi\vert_0)X)y\rangle
=
\tfrac{d}{dt}\big\vert_{t=0}
\omega\bigl((I+tX)x,(I+tX)y\bigr)
=
\omega(Xx,y)+\omega(x,Xy).
\end{equation}
Given $T \in \mathsf{H}^2(\K)$, consider the operator $X_T \in \Hom(\W, \W^{\perp})$ defined by $X_T x = 0$ for $x \in \E$ and $X_T x = \tfrac12 B_\omega T x$ for $x \in \K$, where $B_\omega$ is the unique operator satisfying $\omega = \langle \, \cdot \,, B_\omega \,\cdot \,\rangle$. It is well-defined since for all $w \in \W$ and $x \in \K$, $\langle w, X_T x\rangle = \tfrac12\langle w, B_\omega Tx\rangle = \tfrac12\omega(w,Tx) = 0$, so $X_T x \in \W^\perp$. For all $x$, $y \in \K$,
\[
\omega(X_T x, y) + \omega(x, X_T y) = \tfrac12\langle Tx, y\rangle + \tfrac12\langle x, Ty\rangle = \langle x, Ty\rangle.
\]
Therefore, $ d\Psi\vert_0 X_T = T$, i.e., $d \Psi \vert_0$ is surjective. For $\mathcal{I}$ sufficiently small, $\Psi$ is a submersion on $\mathcal{I}$.

Applying Lemma~\ref{lem:embed} to $\K$ gives the Whitney~(B) stratification
\[
\mathsf{H}^2(\K) = \bigsqcup_{a+b+c=r} \mathsf{H}^2(\K)^{(a,b,c)}.
\]
Since the preimage of a Whitney~(B) stratification under a submersion is again a Whitney~(B) stratification \cite[Section~1.3]{GM1988}, the partition
\[
\mathcal{I} = \bigsqcup_{a+b+c=r} \Psi^{-1}\bigl(\mathsf{H}^2(\K)^{(a,b,c)}\bigr)
\]
is a real-analytic Whitney~(B) stratification. Now the map
\begin{equation}\label{eq:chi}
    \chi \colon \Hom(\W, \W^{\perp}) \to \Gr_k(\F^d), \qquad L \mapsto (I+L)\W,
\end{equation}
defines a chart of the Grassmannian on a sufficiently small neighborhood of $0$. By \eqref{eq:shur-gr} and Sylvester's law of inertia,
\[
\mathsf{s}(\omega\vert_{\chi(L)}) = \mathsf{s}( \Phi(L)) = \mathsf{s}(A(L)) + \mathsf{s}(\Psi(L)) = (p,q,0) + \mathsf{s}(\Psi(L)).
\]
So $\chi(L) \in \mathcal{O}_{(p+a,q+b,c)}$ if and only if $\Psi(L) \in \mathsf{H}^2(\K)^{(a,b,c)}$. Therefore,
\begin{equation}\label{eq:local}
    \chi^{-1}\bigl(\chi(\mathcal{I}) \cap \mathcal{O}_{(p+a,q+b,c)}\bigr) = \Psi^{-1}\bigl(\mathsf{H}^2(\K)^{(a,b,c)}\bigr),
\qquad a+b+c = r.
\end{equation}
As $\chi$ is a real-analytic diffeomorphism onto $\chi(\mathcal{I})$ and Whitney~(B) is a diffeomorphism invariant, the partition
\[
\chi(\mathcal{I}) = \bigsqcup_{a+b+c = r} \bigl(\chi(\mathcal{I}) \cap \mathcal{O}_{(p+a,q+b,c)}\bigr)
\]
is a real-analytic Whitney~(B) stratification. This verifies Whitney~(B) condition.

It remains to verify the frontier condition. Let $\W \in \mathcal{O}_{(p,q,r)}$. Since $\chi(\mathcal{I})$ is an open neighborhood of $\W$ in $\Gr_k(\F^d)$, $\W \in \overline{\mathcal{O}_{(p',q',r')}}$ if and only if $\W \in \overline{\mathcal{O}_{(p',q',r')} \cap \chi(\mathcal{I})}$.
By Lemma~\ref{lem:embed}, the condition is equivalent to $p' \geq p$ and $q' \geq q$ and hence
\[
\mathcal{O}_{(p,q,r)} \cap \overline{\mathcal{O}_{(p',q',r')}} \ne \varnothing
\quad \Longrightarrow \quad
\mathcal{O}_{(p,q,r)} \subseteq \overline{\mathcal{O}_{(p',q',r')}}.
\]
The reverse implication is immediate since each stratum is nonempty, and we obtain the frontier condition.

Taking $(a,b,c) = (0,0,r)$ in the correspondence $\chi^{-1}\bigl(\chi(\mathcal{I}) \cap \mathcal{O}_{(p+a,q+b,c)}\bigr) = \Psi^{-1}\bigl(\mathsf{H}^2(\K)^{(a,b,c)}\bigr)$, we obtain $\chi^{-1}\bigl(\chi(\mathcal{I}) \cap \mathcal{O}_{(p,q,r)}\bigr) = \Psi^{-1}(0)$. Since $\Psi$ is a submersion and $\chi$ a real-analytic diffeomorphism onto the open set $\chi(\mathcal{I}) \subseteq \Gr_k(\F^d)$, we have 
$\codim_\R \mathcal{O}_{(p,q,r)} = \dim_\R \mathsf{H}^2(\K) = \dim_\R \mathsf{H}^2(\F^r)$.
\end{proof}

\subsection{Normal bundle on strata}

As we just saw, the highest-dimensional strata in the stratification \eqref{eq:partition-grass} of the Grassmannian are all indefinite Grassmannians  \eqref{eq:partition-grass1}. Here we will say a few words about the lower-dimensional strata, which may be viewed as ``degenerate Grassmannians'' and could be interesting in their own right.

These lower-dimensional strata $\mathcal{O}_{(p,q,r)}$ are closely related to the vanishing of $\omega$. Since they naturally occur as embedded submanifolds of $\Gr_k(\F^d)$, it is fitting to begin by looking at their normal bundles. We will write
\[
\pi_{\mathcal{T}_k}\colon \mathcal{T}_k
\coloneqq
\bigl\{(\W,x)\in\Gr_k(\F^d)\times\F^d:x\in\W\bigr\}
\to\Gr_k(\F^d),
\quad
(\W,x)\mapsto\W
\]
for the tautological rank-$k$ $\F$-vector bundle over $\Gr_k(\F^d)$. 
For nonnegative integers $p,q,r$ indexing a nonempty stratum $\mathcal{O}_{(p,q,r)}$, we define
\begin{equation}\label{eq:radical-bundle}
\mathcal{K}_{(p,q,r)}
\coloneqq
\bigl\{(\W,x)\in \mathcal{T}_k\vert_{\mathcal{O}_{(p,q,r)}}:
\omega_{m,n}(x,y)=0\text{ for all }y\in\W\bigr\}.
\end{equation}
Thus the fiber of \eqref{eq:radical-bundle} over $\W\in\mathcal{O}_{(p,q,r)}$ is
\[
(\mathcal{K}_{(p,q,r)})_\W
=
\ker(\omega_{m,n}\vert_\W)
=
\W\cap\W^{\perp_{\omega_{m,n}}}.
\]
If $\pi_{\mathcal{E}}\colon\mathcal{E}\to X$ is a Euclidean $\F$-vector bundle, we write
$\mathcal{E}_x\coloneqq\pi_{\mathcal{E}}^{-1}(x)$ and define
\[
\pi_{\mathsf{H}^2(\mathcal{E})}\colon H^2(\mathcal{E})
\coloneqq
\bigl\{(x,H):x\in X,\ H\in\mathsf{H}^2(\mathcal{E}_x)\bigr\}
\to X,
\quad
(x,H)\mapsto x.
\]
For $r=\rank_\F\mathcal{E}$, Lemma~\ref{lem:embed} gives
\[
\rank_\R\mathsf{H}^2(\mathcal{E})
=
\begin{cases}
 r(r+1)/2 &\text{if }\F=\R,\\
 r^2 &\text{if }\F=\C,\\
 r(2r-1) &\text{if }\F=\H,
\end{cases}
\]
which, unsurprisingly, vanishes when $r = 0$. So the results in this section are only nontrivial for the lower-dimensional strata.

\begin{proposition}\label{prop:radical-bundle}
For every nonempty stratum $\mathcal{O}_{(p,q,r)}$, the projection
\[
\mathcal{K}_{(p,q,r)}\to\mathcal{O}_{(p,q,r)},
\quad
(\W,x)\mapsto \W,
\]
is a smooth rank-$r$ $\F$-vector subbundle of $\mathcal{T}_k\vert_{\mathcal{O}_{(p,q,r)}}$.
\end{proposition}

\begin{proof}
For $\W\in\Gr_k(\F^d)$, consider the map $A_\W:\W\to \W$, $x \mapsto P_\W I_{m,n}x$, which defines a smooth bundle endomorphism $\mathcal{T}_k\vert_{\mathcal{O}_{(p,q,r)}} \to \mathcal{T}_k\vert_{\mathcal{O}_{(p,q,r)}}$.
For $x,y\in\W$,
$\langle y,A_\W x\rangle
=
\langle y,I_{m,n}x\rangle
=
\omega_{m,n}(y,x)$. 
Since $\omega_{m,n}$ is Hermitian, $A_\W x=0$ if and only if $\omega_{m,n}(x,y)=0$ for every $y\in\W$. Hence
\[
\ker A_\W
=
\ker(\omega_{m,n}\vert_\W)
=
(\mathcal{K}_{(p,q,r)})_\W.
\]
For every $\W\in\mathcal{O}_{(p,q,r)}$, $\rank (A_\W) = k-r$. If $r=k$, then $A_\W=0$ for every $\W\in\mathcal{O}_{(p,q,r)}$ and $\mathcal{K}_{(p,q,r)}=\mathcal{T}_k\vert_{\mathcal{O}_{(p,q,r)}}$, and we are done.

So suppose $r<k$. Let $\W_0\in\mathcal{O}_{(p,q,r)}$, $\K_0\coloneqq\ker A_{\W_0}$, and $\E_0\coloneqq\K_0^\perp\cap\W_0$. For $x,y\in\W_0$,
\[
\langle x,A_{\W_0}y\rangle
=
\langle x,P_{\W_0}I_{m,n}y\rangle
=
\langle x,I_{m,n}y\rangle
=
\langle P_{\W_0}I_{m,n}x,y\rangle
=
\langle A_{\W_0}x,y\rangle,
\]
so $A_{\W_0}$ is Hermitian. Consequently, $\im(A_{\W_0}) = \ker(A_{\W_0})^\perp = \E_0$. 
It follows that $A_{\W_0}\vert_{\E_0}\colon\E_0\to \E_0$ is a vector space isomorphism and
\[
\dim_\F\E_0=k-r,
\qquad
\dim_\F\K_0=r,
\qquad
\W_0=\E_0\oplus\K_0.
\]

The restriction of the tautological bundle to $\chi(\mathcal{I})$ admits
the smooth trivialization
\[
\Phi\colon\mathcal{T}_k\vert_{\chi(\mathcal{I})}
\to
\chi(\mathcal{I})\times\W_0, \quad 
\bigl(\chi(L),(I+L)x\bigr)
\mapsto
\bigl(\chi(L),x\bigr).
\]
For $L\in\mathcal{I}$, we define $M(L)\in\End(\W_0)$ by
\begin{equation}\label{eq:local-matrix-radical}
A_{\chi(L)}\bigl((I+L)x\bigr)
=
(I+L)M(L)x,
\qquad x\in\W_0.
\end{equation}
Since $P_{\chi(L)} = (I+L)(I+L^*L)^{-1}(I+L)^*$ and $P_{\W_0}(I+L)=I_{\W_0}$,
\[
M(L)
=
(I+L^*L)^{-1}(I+L)^*I_{m,n}(I+L).
\]
Consequently, $M\colon\mathcal{I}\to\End(\W_0)$ is real-analytic and $M(0)=A_{\W_0}$. Now partition $M(L)$ in block form with respect to
$\W_0=\E_0\oplus\K_0$:
\[
M(L)
=
\begin{bmatrix}
M_{11}(L)&M_{12}(L)\\
M_{21}(L)&M_{22}(L)
\end{bmatrix},
\]
with $M_{11}(L)\in\End(\E_0)$, $M_{12}(L)\in\Hom(\K_0,\E_0)$, $M_{21}(L)\in\Hom(\E_0,\K_0)$, and $M_{22}(L)\in\End(\K_0)$. 
Since $A_{\W_0}\vert_{\E_0}$ is an isomorphism and
$A_{\W_0}\vert_{\K_0}=0$, we have
\[
M_{11}(0)=A_{\W_0}\vert_{\E_0}\in\GL(\E_0),
\qquad
M_{12}(0)=M_{21}(0)=M_{22}(0)=0.
\]
Since $\GL(\E_0)$ is open in $\End(\E_0)$, $\mathcal{I}'
\coloneqq
\bigl\{L\in\mathcal{I}:M_{11}(L)\in\GL(\E_0)\bigr\}$ is an open neighborhood of $0$. Restricting $\mathcal{I}$ to $\mathcal{I}'$,
we may assume that $M_{11}(L)$ is invertible for every $L\in\mathcal{I}$.

Let $\mathcal{U}\coloneqq
\chi(\mathcal{I})\cap\mathcal{O}_{(p,q,r)}$. For every $L\in\chi^{-1}(U)$,
\eqref{eq:local-matrix-radical} gives that $ \rank_\F M(L) = \rank_\F A_{\chi(L)} = k-r$. By the same Schur complement calculation used in earlier proofs,
\[
S(L)
\coloneqq
M_{22}(L)-M_{21}(L)M_{11}(L)^{-1}M_{12}(L)= 0
\in\End(\K_0)  \quad \text{for all } L\in\chi^{-1}(\mathcal{U}).
\]
Therefore
\[
\ker M(L)
=
\biggl\{
\begin{bmatrix}
-M_{11}(L)^{-1}M_{12}(L)z\\
z
\end{bmatrix}
:z\in\K_0
\biggr\} \cong \K_0
\]
and thus
\[
\Theta\colon \mathcal{U}\times\K_0 \to
\mathcal{K}_{(p,q,r)}\vert_{\mathcal{U}}, \quad
\bigl(\chi(L),z\bigr) \mapsto
\biggl(\chi(L),(I+L)\begin{bmatrix}
-M_{11}(L)^{-1}M_{12}(L)z\\
z
\end{bmatrix} \biggr)
\]
is well-defined and clearly smooth. For every $L\in\chi^{-1}(\mathcal{U})$, the fiber map $\restr{\Theta}{\{\chi(L)\}\times\K_0}
\colon
\K_0\to
(\mathcal{K}_{(p,q,r)})_{\chi(L)}$ is an $\F$-linear isomorphism with inverse
\[
\Theta^{-1}
\left(
\chi(L),
(I+L)
\begin{bmatrix}u\\z\end{bmatrix}
\right)
=
\bigl(\chi(L),z\bigr),
\qquad
\begin{bmatrix}u\\z\end{bmatrix}\in\ker M(L),
\]
and is clearly smooth. This completes the proof as $\W_0$ is arbitrary.
\end{proof}

Recall that for an embedded smooth submanifold $S\subseteq\mathcal{M}$, its normal bundle is the real vector bundle
\[
\nu(S)
\coloneqq
\mathbb{T}\mathcal{M}\vert_S/\mathbb{T}S.
\]
We now determine the normal bundle of every stratum in \eqref{eq:partition-grass}, with $\mathcal{K}_{(p,q,r)}$ as defined in \eqref{eq:radical-bundle}.

\begin{theorem}[Normal bundle]
\label{thm:normal-bundle}
There is a short exact
sequence of real vector bundles over $\mathcal{O}_{(p,q,r)}$,
\begin{equation}\label{eq:normal-exact}
0\to \mathbb{T} \mathcal{O}_{(p,q,r)}
\to \mathbb{T}\Gr_k(\F^d)\vert_S
\xrightarrow{\mathcal{D}}
\mathsf{H}^2( \mathcal{K}_{(p,q,r)})
\to0.
\end{equation}
Consequently,
\[
\nu (\mathcal{O}_{(p,q,r)})
\cong
\mathsf{H}^2(\mathcal{K}_{(p,q,r)}).
\]
\end{theorem}

\begin{proof}
Let $\W\in \mathcal{O}_{(p,q,r)}$ and $\K\coloneqq\ker(\omega_{m,n}\vert_\W)$. We identify $\mathbb{T}_\W\Gr_k(\F^d)\cong\Hom(\W,\W^\perp)$
via $L\mapsto(I+L)\W$. Set $\mathcal{D}_\W(X)
\coloneqq
P_\K\bigl(X^*I_{m,n}+I_{m,n}X\bigr)\vert_\K
\in\mathsf{H}^2(\K)$. Then
\[
\langle x,\mathcal{D}_\W(X)y\rangle
=
\langle x,X^*I_{m,n}y\rangle
+\langle x,I_{m,n}Xy\rangle =
\omega_{m,n}(Xx,y)+\omega_{m,n}(x,Xy)
\]
for all $x,y\in\K$.
By \eqref{eq:diff}, $\mathcal{D}_\W = d\Psi \vert_0$. We next check that $\mathcal{D}_\W$ induces a smooth vector-bundle morphism. By Proposition~\ref{prop:radical-bundle}, $\mathcal{K}_{(p,q,r)}$ is a smooth subbundle of $\mathcal{T}_k\vert_{\mathcal{O}_{(p,q,r)}}$. If $v_1,\dots,v_r$ is a local smooth orthonormal frame of $\mathcal{K}_{(p,q,r)}$, then $P_\K=\sum_{i=1}^r v_i v_i^*$, so $\W\mapsto P_\K$ is smooth. Since $\W\mapsto P_\W$ is real-analytic, $\mathcal{D}_\W(X)$ depends smoothly on $\W$ and real-linearly on $X$. Hence $\mathcal{D}_\W$ induces a smooth real vector-bundle morphism $\mathcal{D}:
\mathbb{T}\Gr_k(\F^d)\vert_{\mathcal{O}_{(p,q,r)}}
\to
\mathsf{H}^2(\mathcal{K}_{(p,q,r)})$.

Fix $\W\in \mathcal{O}_{(p,q,r)}$. Let $\chi$ and $\Psi$ be defined by \eqref{eq:chi} and \eqref{eq:psi} respectively. By shrinking $\mathcal{I}$ if necessary, $\chi^{-1}\bigl(\chi(\mathcal{I})\cap S\bigr)=\Psi^{-1}(0)$
by \eqref{eq:local}. Since $\Psi$ is a submersion, $\mathcal{D}_\W$ is surjective. Therefore,
\[
(d\chi\vert_0)^{-1}(\mathcal{T}_\W \mathcal{O}_{(p,q,r)})
=
\mathcal{T}_0\Psi^{-1}(0)
=
\ker(d\Psi\vert_0)
=
\ker\mathcal{D}_\W .
\]
As $\W\in \mathcal{O}_{(p,q,r)}$ is arbitrary, \eqref{eq:normal-exact} is exact on
every fiber and is therefore a short exact sequence of smooth real
vector bundles.
\end{proof}

\section{Conclusion}

This is likely the first systematic study of the  indefinite Grassmannian. We show that it is an interesting object both geometrically and topologically, not least because of their relations, both geometric and topological, to the standard Grassmannian. Given that this is a first foray, it is unsurprising that there are many open questions (e.g., the integral cohomology left unaddressed Corollary~\ref{cor:cohom}) and future directions (e.g., the indefinite analogs of the symplectic twistor spaces in \cite{symplectic}), which we hope would generate exciting further works for interested readers and ourselves.

\subsection*{Acknowledgments} This work is partially supported by the Vannevar Bush Faculty Fellowship ONR N000142312863. 

\bibliographystyle{abbrv}

\end{document}